\documentclass[12pt]{amsart}
\usepackage{amssymb,amsthm,amsmath,amstext}
\usepackage{mathrsfs}  
\usepackage{bm}        
\usepackage{mathtools} 

\usepackage[left=1.28in,top=.8in,right=1.28in,bottom=.8in]{geometry}

\usepackage{graphicx}
\usepackage[all]{xy}

\theoremstyle{plain}
\newtheorem{theorem}{Theorem}[section]
\newtheorem{prop}[theorem]{Proposition}
\newtheorem{lemma}[theorem]{Lemma}

\newtheorem{cor}[theorem]{Corollary}

\theoremstyle{definition}

\newtheorem{remark}[theorem]{Remark}
\newtheorem{assumptions}[theorem]{Assumptions}
\newcommand{\sheaf}[1]{\mathscr{#1}}

\renewcommand{\AA}{\sheaf{A}}

\newcommand{\PP}{\sheaf{P}}
\newcommand{\XX}{\sheaf{X}}
\newcommand{\YY}{\sheaf{Y}}

\newcommand{\QQ}{\sheaf{Q}}

\newcommand{\Brl}{{}_\ell\mathrm{Br}}
\newcommand{\Brn}{{}_n\mathrm{Br}}

\newcommand{\Z}{\mathbb Z}

\DeclareMathOperator{\Br}{\mathrm{Br}}

\usepackage{hyperref}
\usepackage{color}

\hypersetup{colorlinks=true,linkcolor=blue,anchorcolor=blue,citecolor=blue,letterpaper=true}

\begin{document}

\title[ Third Galois Cohomology group]
{ Third Galois Cohomology group of function fields of curves over number fields} 

\author[Suresh]{V.\ Suresh}
\address{Department of Mathematics \& Computer Science \\ %
Emory University \\ %
400 Dowman Drive~NE \\ %
Atlanta, GA 30322, USA}
\email{suresh@mathcs.emory.edu}

\date{}

\begin{abstract} Let $K$ be a  number field or a $p$-adic field 
and $F$ the function field of a curve over $K$. Let $\ell$ be a prime. 
Suppose that $K$ contains a primitive $\ell^{\rm th}$ root of unity.
If $\ell =  2$ and $K$ is a number field,  then assume that $K$  is a totally imaginary number field.
In this article we show that  every element in  $H^3(F, \mu_\ell^{\otimes 3})$ is a symbol.  
This leads to the finite generation of  the Chow group of zero-cycles on a quadric fibration   of a curve
over a totally imaginary number field.   
\end{abstract} 
 
\maketitle

\def\ZZ{${\mathbf Z}$}
\def\ih{${\mathbf H}$}
\def\RR{${\mathbf R}$}
\def\IF{${\mathbf F}$}
\def\QQ{${\mathbf Q}$}
\def\IP{${\mathbf P}$}

\section{Introduction} 
Let $F$ be a field and $\ell$ a prime not equal to the  characteristic of $F$.
For $n \geq 1$, let $H^n(F, \mu_\ell^{\otimes n})$ be the $n^{\rm th}$ Galois cohomology 
group with coefficients in $\mu_\ell^{\otimes n}$.  We have $  F^*/F^{*\ell} \simeq H^1(F, \mu_\ell)$.
For $a \in F^*$, let $(a ) \in H^1(F, \mu_\ell)$ denote the image of the class of $a$ in $F^*/F^{*\ell}$.
Let $a_1, \cdots , a_n \in F^*$. The cup product   $( a_1 ) \cdots (a_n ) 
\in H^n(F, \mu_\ell^{\otimes n})$ is called a {\it symbol}. A theorem of
 Voevodsky (\cite{Voe}) asserts that every element in $H^n(F,
\mu_\ell^{\otimes n})$ is a sum of symbols. 
Let $\alpha \in H^n(F, \mu_\ell^{\otimes n})$. The {\it symbol length} of $\alpha$ 
 is defined as the smallest $m$ such that $\alpha $
is a sum of $m$ symbols in $H^n(F, \mu_\ell^{\otimes n})$. 

A consequence of class field theory is that if $K$ is a global  field or a local field, then 
every element in $H^n(K, \mu_\ell^{\otimes n})$ is a symbol.

Let $K$ be a $p$-adic field and   $F$   the function of a curve over   $K$. 
Suppose that $K$ contains a primitive $\ell^{\rm th}$ root of unity.
If $\ell \neq p$, then it was proved in (\cite{Su1}, cf. \cite{Eric}) that 
 the symbol length of  every element in $H^2(F, \mu^{\otimes 2}_\ell)$ is at most 2. 
 If  $p \neq \ell$,  then it was proved in (\cite{PS2}, cf. \cite{PS4}) that every element in 
 $H^3(F, \mu_\ell^{\otimes 3})$ is a symbol. 
If $\ell = p$, then it was proved in (\cite{PS3}) that  for every central simple 
 algebra $A$ over $F$, the index of $A$ divides the square of the period of $A$.
 In particular if  $p = 2$,  then  the symbol length of every element in  $H^2(F, \mu^{\otimes 2}_2)$ is 
 at most 2.  Since $u(F) = 8$ (\cite{HB}, \cite{L},  cf. \cite{PS3}), it follows that every element in 
$H^3(F, \mu_2^{\otimes 3})$ is a symbol.

If $F$ is the function field of a curve over a global field of
positive characteristic $p$, $\ell \neq p$ and $F$ contains a primitive $\ell^{\rm th}$ root of unity,  then 
it was proved in (\cite{PS4}) that every element in $H^3(F, \mu_\ell^{\otimes 3})$ is a symbol.  

Let $K$ be a    number field and $F$ the function field of   a curve over $K$.
In (\cite{Su}), it was proved that given finitely any elements $\alpha_1, \cdots , \alpha_ r \in H^3(F, \mu_2^{\otimes 3})$, 
there exists $f \in F^*$ such that $\alpha_i = (f) \cdot \beta_i$ for some $\beta_i \in H^2(F, \mu_2^{\otimes 2})$. 
In particular if there exists an integer $N$ such that  the symbol length of every element 
in $H^2(F, \mu_2^{\otimes 2})$ is bounded by $N$, then the symbol length of every element in $H^3(F, \mu_2^{\otimes 3})$
is bounded by $N$. In (\cite{LPS}), it was proved that such an integer  $N$ exists under the hypothesis that a conjecture of 
Colliot-Th\'el\`ene  on the existence of 0-cycles of degree 1 holds. 
However it is still open whether such $N$ exists.

In this paper we prove the following (see \ref{imaginary})

\begin{theorem} Let $K$ be a  global field   or a local field
and $F$ the function field of a curve over $K$. Let $\ell$ be a prime not equal to char$(K)$. 
Suppose that $K$ contains a primitive $\ell^{\rm th}$ root of unity.
If $\ell \neq 2$ or $K$ is a local field or $K$ is a global field with  no real orderings, 
 then every element in  $H^3(F, \mu_\ell^{\otimes 3})$ is a symbol.  \end{theorem}

The above theorem for  $K$ a $p$-adic field and $\ell \neq p$ is proved in  (\cite{PS2}, cf. \cite{PS4}).
Our method in this  paper is uniform, it  covers both global and local fields at the same time, without the assumption  
$\ell \neq p$.

We have the following (see \ref{rank10})

\begin{cor} Let $K$ be a totally imaginary  number field and $F$ the function field of a  curve 
over $K$. Let $q$ be  a quadratic form over $F$ and $\lambda \in F^*$.
If the dimension of $q$ is at least 5, then $q \otimes <\!1, -\lambda \!>$ is isotropic. 
\end{cor}

For a smooth integral variety  $X$ over a field $k$, let $CH_0(X)$ be the Chow group of 0-cycles modulo 
rational equivalence. If $k$ is  a number field and $X$ a smooth projective geometrically integral   curve, 
then   Mordell-Weil theorem asserts that $CH_0(X)$ is finitely generated.
 
 Let $C$ be a smooth projective geometrically integral curve over a field
$k$. Let $X \to C$ be a (admissible) quadric fibration (cf. \cite{CTSk}). Let $CH_0(X/C)$ be the kernel of the natural 
homomorphism $CH_0(X) \to CH_0(C)$.  If char$(k) \neq 2$, Colliot-Th\'el\`ene and Skorobogatov   identified  $CH_0(X/C)$ with 
a certain sub-quotient of $k(C)^*$ (\cite{CTSk}). From this identification it follows that $CH_0(X/C)$ is a 2-torsion group.
Thus $CH_0(X/C)$ is finitely generated if and only if it is finite. 
Suppose that $k$ is a number field. If dim$(X) \leq 2$, then the finiteness of $CH_0(X/C)$ is a result of Gros (\cite{Gros}).
If dim$(X) = 3$, then it was proved in (\cite{CTSk}, \cite{PS}) that $CH_0(X/C)$ is finite.
Thus for dim$(X) \leq 3$, $CH_0(X)$ if finitely generated. 
Colliot-Th\'el\`ene and Skorobogatov conjectured that if dim$(X) \geq 4$ and $k$ is a totally imaginary  number field, 
 then $CH_0(X/C) = 0$.
As a consequence of our main theorem, we prove the  following (see \ref{zero-cycles}).

\begin{cor} Let $K$ be a totally imaginary  number field, $C$ a smooth projective geometrically integral curve 
over $K$. Let $X \to C$ be an admissible quadric fibration.  If dim$(X) \geq 4$, then $CH_0(X/C) = 0$.
In particular $CH_0(X)$ is finitely generated.  
\end{cor}

Let $K$ be a local field and $p$ the characteristic of the residue field of $K$ or  a global field of
positive characteristic $p$.  Let  $F$ be   the  function field of a curve over $k$   and
$\ell$ a prime not equal to $p$.  Let us recall  that  the main ingredient in the proof of the fact that every 
element in $H^3(F, \mu_\ell^{\otimes 3})$ is a symbol (\cite{PS2})  is a certain local-global principle for divisibility of an element of 
$H^3(F, \mu_\ell^{\otimes 3})$ by a symbol  in $H^2(F, \mu_\ell^{\otimes 2})$ (\cite{PS2}, \cite{PS4}). In fact it was proved that
for a given $\zeta \in H^3(F, \mu_\ell^{\otimes 3})$ and  a symbol $\alpha \in H^2(F, \mu_\ell^{\otimes 2})$ if  for every 
discrete valuation $\nu$ of $F$ there exists $f_\nu \in F^*$ such that $\zeta - \alpha \cdot (f_\nu)$ is unramified at $\nu$, 
then there exists $f\in F^*$ such that $\zeta = \alpha \cdot (f)$. 
In the proof of  this local-global principle,   the existence of residue homomorphisms on $H^2(F, \mu_\ell^{\otimes 2})$ and
$H^3(F, \mu_\ell^{\otimes 3})$ is used. However note that if $K$ is a global field or a $p$-adic field with $\ell  = p$, 
then   there is no `residue homomorphism' on $H^2(F, \mu_\ell^{\otimes 2})$ which describes the unramified Brauer group. 
 
We now briefly explain the main ingredient of our result.
Let $K$ be a global field  or a local field and $F$ the function field of a curve over $K$.
Let $\ell$ be a prime not equal to characteristic of $K$. Suppose that $K$ contains a primitive $\ell^{\rm th}$ root
of unity. Let $\nu$ be a discrete valuation on $F$ and   $\kappa(\nu)$   the residue field at $\nu$. 
 Then there is a   residueue homomorphism $H^3(F, \mu_\ell^{\otimes 3}) \to \Brl(\kappa(\nu))$ (\cite[\S 1]{K2}).
Let $\zeta \in   H^3(F, \mu_\ell^{\otimes 3})$ and   $\alpha = (a, b) \in  H^2(F, \mu_\ell^{\otimes 2})$.    
First we show that if there is a   regular proper model $\XX$ of $F$ such that 
the triple $(\zeta, \alpha, \XX)$ satisfies    certain assumptions,  then there is a local global principle for the 
divisibility of $\zeta$ by $\alpha$  (cf. \ref{localglobal}). One of the key assumptions is that 
$a \in F^*$ has some `nice'  properties at closed   points of $\XX$ which are on  the support of the prime $\ell$ and
in the  ramification of $\zeta$ or $\alpha$ (cf. \ref{assumptions}, \ref{assumptions2}). 
These  assumptions on $a$ enable us to avoid the residue homomorphisms on 
$H^2(F, \mu_\ell^{\otimes 2})$.  Let  $\zeta \in H^3(F, \mu_\ell^{\otimes 3})$.
Suppose that $\ell \neq 2$ or $K$ is a local field or $K$ is a global field without real orderings.
Then we show that there exist a symbol  $\alpha = [a, b) \in H^2(F, \mu_\ell^{\otimes 2})$ and a regular proper model 
$\XX$ of $F$ satisfying the assumption of \S 6.   Thus, by the local global principle for the divisibility, there exists 
$f \in F^*$ such that $\zeta - \alpha \cdot (f)$ is unramified on $\XX$.  
  Then,  using a result of Kato   (\cite{K2}), 
 we arrive at the proof of  our main result (\ref{main}).
   

 \section{Preliminaries }
 
 Let $K$ be a field and  $\ell$   a prime. 
 Then every  non-trivial element  in $H^1(K, \Z/\ell)$ is represented by a pair $(L, \sigma)$, where 
 $L/K$ is a cyclic field extension of degree $\ell$ and $\sigma$ a generator of Gal$(L/K)$.
 
 Suppose $\ell \neq $ char$(K)$ and $K$ contains a primitive 
 $\ell^{\rm th}$ root of unity. Fix a primitive $\ell^{\rm th}$ root of unity $\rho \in K$.
 Let $L/K$ be a cyclic extension of degree $\ell$. 
  Then, by Kummer theory, we have $L = K(\sqrt[\ell]{a})$ for some $a \in K^*$
and $\sigma \in $ Gal$(L/K)$   given by 
 $\sigma(\sqrt[\ell]{a})  = \rho \sqrt[\ell]{a}$  is a  generator   of   Gal$(L/K)$.
  Thus we have an isomorphism $K^*/K^{*\ell} \to H^1(K, \Z/\ell\Z)$
 given by sending  the class of  $a $ in $K^*/K^{*\ell}$ to the pair $(L, \sigma)$, 
 where $L = K[X]/(X^\ell - a)$  and  $\sigma(\sqrt[\ell]{a}) =\rho\sqrt[\ell]{a}$.
 Let $a \in K^*$.   If  the image of the class of  $a$  in $H^1(F, \Z/\ell \Z)$ is $(L, \sigma)$ and $i$ is coprime to $\ell$, 
 then the image of $a^i$ is $(L, \sigma^i)$. In particular $(L, \sigma)^i = (L, \sigma^i)$ for all $i$ coprime to $\ell$. 
 
 Suppose char$(K) = \ell$  and $L/K$ is a cyclic extension of degree $\ell$. 
 Then, by Artin-Schreier theory,   $L = K[X]/(X^\ell - X + a)$
 for some $a \in K$. The element   $\sigma \in  $ Gal$(L/K)$ given by $\sigma(x) = x + 1$, 
 where $x \in L$ is the image of $X$ in $L$, is  a generator of Gal$(L/K)$. 
 Let $\wp : K \to K$ be the Artin-Schreier map $\wp(b) = b^\ell - b$. 
 We have an isomorphism $K/\wp(K) \to H^1(K, \Z/\ell\Z)$ given by sending the class of $a$ to 
 the pair $(L,  \sigma)$, where  $L = K[X]/(X^\ell - X + a)$ and $\sigma(x) = x + 1$.
 We note that if the image  the class of $a$ is    $(L, \sigma)$, then the image of the class of 
 $ia$ is $(L, \sigma^i)$ for all $1 \leq i \leq \ell -1$. 
 
 In either   case  (char$(K) \neq \ell $ or char$(K)  = \ell$), for $a \in K^*$ (or $K$), 
the pair $(L, \sigma)$ is denoted by $[ a )$. Some times, by  abuse of the notation, we also 
 denote the cyclic extension $L$ by $[a )$. 
 
 Let $R$ be a regular local ring  with field of fractions $K$, maximal ideal $m_R$ and residue field $\kappa$. 
 Let $L/K$ be a finite separable   extension and $S$ the integral closure of $R$ in $S$.
 We say that $L/K$ is {\it unramified} at $R$
 if  $S/m_RS$ is a  separable $\kappa$-algebra.  Let $R$ be a regular ring 
 with field of fractions $K$ and $L/K$ a finite separable extension. We say that $L/K$ is {\it unramified} at
 a prime ideal $P$ of $R$ if $L/K$ is unramified at the local ring  $R_P$ of $R$ at $P$.
 We say that $L/K$ is {\it unramified}  on $R$ if $L/K$ is unramified at every prime ideal of $R$. 
 If $L/K$ is unramified at a prime ideal $P$ of $R$, $S_P$ denotes the integral closure of $R_P$
 and the separable $R_P/PR_P$-algebra $S_P/PS_P$ is called the {\it residue field} of $L$ at $P$. 
 Note that $S_P/PS_P$ is a product of separable field extensions of  $R_P/PR_P$.
 If $R$ is a regular local ring, then $L/K$ is unramified at $R$ if and only if the discriminant of $L/K$ is 
 a unit in $R$ (cf. \cite[Exercise 3.9, p.24]{Milne}). 
 Thus in particular, $L/K$ is unramified on $R$ if and only is $L/K$ is unramified at 
 all height one prime ideals of $R$. 
  If $L$ is a product of fields $L_i$ with $K \subset L_i$, then we say that $L/K$ is {\it unramified} on $R$ if each $L_i/K$
is unramified on $R$.

 \begin{lemma}
 \label{rho} (\cite[Proposition 4.2.1]{CT-kato}) Let $K$ be a field with a discrete   valuation $\nu$  and  $\kappa$  
 the  residue field at $\nu$. Let $m$ be the maximal ideal of the valuation ring $R$ at $\nu$. 
 Suppose that char$(K) = 0$ and char$(\kappa) = \ell > 0$.  Suppose that $K$ contains a primitive 
 $\ell^{\rm th}$ root of unity $\rho$. Then  $\ell = x(   \rho - 1)^{\ell -1}$ for some   unit $x$  at $\nu$ with 
 $x \equiv -1 $ modulo $m$.  In particular $\nu(\rho - 1) = \frac{\nu(\ell)}{\ell -1}$. 
 \end{lemma}
 
 \begin{proof}  Let $\eta = \rho-1$.   Since  $\eta$ is a zero of  the polynomial 
 $(X + 1)^{\ell -1} + (X + 1)^{\ell -2} + \cdots + 1$ and the binomial coefficients $^\ell C_i$, $1 \leq i \leq \ell-1$
 are divisible by $\ell$, we have 
 $\eta^{\ell-1} = -\ell (\eta^{\ell-2} + b_{\ell-3}\eta^{\ell-3} + \cdots + b_1\eta + 1)$ for some 
 $b_i \in R$.   Since $\rho \equiv 1$ modulo $m$,  $\eta \in m$ and 
 hence $y = \eta^{\ell-2} + b_{\ell-3}\eta^{\ell-3} + \cdots + b_1\eta + 1
 $ is a unit in $R$ and $y  \equiv 1$ modulo $m$. Then $x = -y^{-1}$ has the required property.   
 \end{proof}

\begin{lemma}
 \label{epp1} Suppose  $R$ is  a discrete valuation ring
 with field of fractions $K$ and residue field $\kappa$.
 Suppose that char$(K) = 0$, char$(\kappa) = \ell > 0$ and $K$ contains  a primitive $\ell^{\rm th}$ root of 
 unity $\rho$.  
 Let  $u \in R$ and $\overline{u} \in \kappa$ the image of $u$. 
 If $1 - u(\rho-1)^\ell \in R^\ell$, then $X^\ell - X + \bar{u}$  has a root in  $\kappa$.  
\end{lemma}

\begin{proof} Let  $m$ be  the maximal ideal of $R$. 
Suppose that $u \in m$. Then $\bar{u} = 0$ and  $X^\ell - X$ has a root in $\kappa$. 

Suppose that $u \in R$ is a unit.
Suppose $1 - u(\rho-1)^\ell \in R^\ell$. 
  Let $z \in R$  with  $z^ \ell =   1 - u(\rho  -1)^\ell  \in R$. Since $\rho -1 \in m$, 
 $1 - u(\rho  -1)^\ell$ is a unit in $R$ and hence $z$ is a unit in $R$ with    $z^\ell \equiv 1$ modulo $m$. 
 Since char$(\kappa) = \ell$, 
 $z \equiv  1 $ modulo $m$. Thus  $z = 1 + d$ for some $d \in m$.  
 Since $z^\ell = (1 + d)^\ell = 1 + \ell d +  \cdots + d^\ell$, 
 all the binomial  coefficients are divisible by $\ell$ and $d \in m$, we have 
 $z^\ell = 1 + \ell dy + d^\ell$ for some unit $y \in R$ with $y \equiv 1 $ modulo $m$. 
 Since  $z^ \ell = 1 -  u(\rho - 1)^\ell$,  we have  $ \ell dy + d^\ell  =   -u(\rho - 1)^\ell$.

 We claim that $\nu(d) = \nu(\rho -1)$.  Suppose that $\nu(\ell d) = \nu(d^\ell)$. Then 
 $\nu(\ell) + \nu(d) = \ell \nu(d)$ and hence $\nu(d) = \frac{\nu(\ell)}{\ell-1} = \nu(\rho -1)$ (\ref{rho}). 
 Suppose that $\nu(\ell d) < \nu(d^\ell)$. Then $\nu(\ell dy + d^\ell) = \nu(\ell d) = \nu(\ell) + \nu(d)$.
 Since $  \ell dy + d^\ell  =   -u(\rho - 1)^\ell$,   $\nu(\ell) + \nu(d) = \ell \nu(\rho -1)$
 and hence $\nu(d) = \ell \nu(\rho -1) - \nu(\ell) = \ell \frac{\nu(\ell)}{\ell -1} - \nu(\ell) = \frac{\nu(\ell)}{\ell -1} = \nu(\rho-1)$. 
 Suppose that $\nu(\ell dy) > \nu(d^\ell)$. Then $\ell \nu(\rho -1) = \nu(d^\ell) =\ell \nu(d)$ and hence 
 $\nu(d) = \nu(\rho -1)$. 
 
 Since $\nu(d) = \nu(\rho -1)$, we have $d = w(\rho -1)$ for some unit $w \in R$. 
 By (\ref{rho}), we have $\ell = x(\rho -1)^{\ell-1}$ with $x \equiv  -1$ modulo $m$.
 Thus $-u(\rho-1)^\ell = \ell d y + d^\ell =  xy w (\rho -1)^\ell + w^\ell (\rho -1)^\ell $
 and hence $-u = w^\ell + xy w$. Since $x \equiv -1$ modulo $m$ and $y \equiv 1 $ modulo $m$, 
 we have $\overline{w}^\ell - \overline{w} + \overline{u} = 0$.  
  In particular $X^\ell - X + \bar{u}$ has a root in   $\kappa$.
\end{proof}

We  have   the following (cf. \cite[Proposition 1.4]{epp})
 
 \begin{prop}
 \label{epp} 
 Suppose  $R$ is  a discrete valuation ring
 with field of fractions $K$ and residue field $\kappa$.
 Suppose that char$(K) = 0$, char$(\kappa) = \ell > 0$ and $K$ contains  a primitive $\ell^{\rm th}$ root of 
 unity $\rho$.  
 Let  $u \in R$ be a unit and  $L  = K[X]/(X^\ell - (1 - u (\rho - 1) ^\ell))$. 
   Let $S$ be the integral closure of 
 $R$ in $L$.   Then $L/K$ is unramified at $R$ 
 and  \\
$\bullet$ if  $X^\ell - X + \overline{u}$ is irreducible in $\kappa[X]$,  then there is a unique maximal ideal
$m_S$ of $S$ with $S/m_S \simeq  
  \kappa[X]/(X^\ell - X + \overline{u})$, where $\overline{u}$ is the image of $u$ in $\kappa$ \\
  $\bullet$ if  $X^\ell - X + \overline{u}$ is reducible in $\kappa[X]$, then the maximal ideal $m_R$ of 
  $R$ splits into product of $\ell$ maximal ideals   of $R$. 
   \end{prop}
 
 \begin{proof}    Without loss of generality we assume that $R$ is complete and $L$ is a field. 
 Then $S$ is a complete discrete valuation  ring.  Let $m_R$ be the maximal ideal of $R$ and
 $m_S$ the maximal ideal of $S$. Then $m_RS = m_S$.

 Suppose that $X^\ell - X + \overline{u}$ is irreducible in $\kappa[X]$.
 Since  $ 1 - u(\rho-1)^\ell \in S$,  by (\ref{epp1}), 
  $X^\ell - X - \overline{u}$ has a root in $S/m_S$. 
 Since $[S/m_S : \kappa]  \leq \ell$,   $S/m_S \simeq \kappa[X]/(X^\ell - x + \overline{u})$ 
 and hence $m_S$ is the unique maximal ideal of  $S$ with  $L/K$   unramified at $R$.

 Suppose that $X^\ell - X + \overline{u}$ is reducible in $\kappa[X]$. 
 Since char$(\kappa) = \ell$, $X^\ell - X + \overline{u}$ has $\ell$ distinct roots in  $\kappa$.
 Since $R$ is complete,  $X^\ell - X + \overline{u}$ has a root $w$ in $R$.
 Let $d = w(\rho-1) \in R$. Then, as in the proof of (\ref{epp1}),  we have 
 $ (1 + d)^ \ell =  1 + \ell d y + d^\ell$ for some $y \in R$ with $y \equiv 1$ modulo $m_R$. 
 By (\ref{rho}), we have $\ell = x(\rho-1)^{\ell-1}$ for some $x \in R$ with $x \equiv -1$ modulo $m_R$. 
 Since $ w^\ell =  w   -u$ and $d = w(\rho-1)$, we have 
 $$
 \begin{array}{rcl}
  (1 + d)^\ell = 1  + \ell d y + d^\ell  & =  &  1 + \ell w(\rho-1) y  + w^\ell (\rho-1)^\ell \\
  &  = & 1 +   \ell w (\rho-1) y  +  w(\rho-1)^\ell - u (\rho-1)^\ell \\
  & = & 1  + x y w (\rho-1)^\ell + w(\rho-1)^\ell  - u(\rho-1)^\ell  \\
  & = & 1 + w (\rho-1)^\ell (xy + 1) - u(\rho-1)^\ell .
  \end{array}
 $$
 Since $xy + 1 \equiv 0$ modulo $m$, we have 
 $(1 + d)^\ell = 1 - u(\rho-1)^\ell$ modulo $(\rho-1)^\ell m$ and hence 
 $1 - u(\rho-1)^\ell \in R^{*\ell}$ (cf. \cite[\S 0.3]{epp})
  \end{proof}

 \begin{cor}
 \label{epp_general}
 Suppose that $A$ is a regular local ring of dimension two  with  field of fractions $F$, 
  maximal ideal  $m$ and 
 residue field $\kappa$. Suppose that char$(F) = 0$,  char$(\kappa) = \ell > 0$ and 
 $F$ contains a primitive $\ell^{\rm th}$ root of  unity $\rho$. 
 Let $u \in A$ be a unit and $L = F[X]/(X^\ell - (1 - u (\rho - 1)^\ell))$. Suppose that 
 $L$ is a field. 
 Let $S$ be the 
 integral closure of $A$ in $L$.  Then $L/F$ is unramified on $A$ and $S/mS \simeq 
 \kappa[X]/(X^p - X + \overline{u})$, where $\overline{u}$ is the image of $u$
in $\kappa$.   
 \end{cor}
 
 \begin{proof} Since char$(\kappa) = \ell$ and $\rho^\ell = 1$, $1 -\rho$ is in the maximal ideal of $A$
 and hence $1 - u (\rho - 1)^\ell$ is a unit in $A$. 
 Let $P$ be a prime ideal of $A$ of height one. 
Suppose char$(A/P) \neq \ell$. 
Since $1 - u (\rho - 1)^\ell$ is a unit in $A$,   $L/K$ is unramified at 
$P$.  If char$(R/P) = \ell$, then  by (\ref{epp}), $L/K$ is unramified at $P$.
Thus $L/K$ is unramified  on $A$. 
 
Let  $m = (\pi, \delta)$ be the  maximal ideal of $A$. 
Since $L/K$ is unramified on $A$, $S/\pi S$ is a regular semi-local ring (cf. \cite[Proposition 3.17, p.27]{Milne}). 
 Suppose that char$(A/(\pi)) \neq \ell$. 
Since $1 - u (\rho - 1)^\ell$ is a unit at $\pi$, $L/K$ is unramified at $\pi$
and $S \otimes_A A_{(\pi)}/(\pi) \simeq  (A_{(\pi)}/(\pi))[X]/(X^\ell - (1 - \overline{u}(\overline{\rho}-1)^\ell))$, 
where $\bar{~}$ denotes the image modulo $(\pi)$. 
 Hence by (\ref{epp}), $S/(\pi, \delta)S = \kappa[X]/(X^\ell - X + \overline{u})$.
 Suppose that char$(A/(\pi)) = \ell$. Then, by (\ref{epp}),  the field of fractions of 
 $S/\pi S$ is the field of fractions of $(A/(\pi))[X]/(X^\ell - X + \overline{u})$.
  Since $u$ is a unit in $A/(\pi)$,
 $A/(\pi)[X]/(X^\ell - X + \overline{u})$ is a regular local ring and hence 
 $S/\pi S \simeq A/(\pi)[X]/(X^\ell - X + \overline{u})$.
 Hence  $S/(\pi, \delta)S = \kappa[X]/(X^\ell - X + \overline{u})$. 
 \end{proof}

Let $R$ be a regular ring  of dimension at most 2 with field of fractions $K$ and $\ell$ a prime.
If $\ell$ is  not equal to char$(K)$, then assume that $K$ contains a primitive 
$\ell^{\rm th}$ root of unity $\rho$.  Suppose   $L = [a)$ is  a cyclic extension  of $K$ of degree $\ell$.  
Let $P$ be a prime ideal of $R$, $\kappa(P) = R_P/PR_P$  and $S_P$ the integral closure of $R_P$ in $L$. 
Suppose char$(\kappa(P)) \neq \ell$.  Then $L = K[X]/(X^\ell - a)$ and hence $S_P/PS_P \simeq \kappa(P)[X]/(X^\ell -\overline{a})$
where $\overline{a}$ is the image of $a$ in $\kappa(P)$. 
Suppose char$(\kappa(P)) = \ell$, char$(K) \neq \ell$ and $a  = 1 - u(\rho-1)^\ell$ for some $u \in R_P$.  
Then, by (\ref{epp}, \ref{epp_general}),  $S_P/PS_P \simeq \kappa(P)[X]/(X^\ell -X + \overline{u})$.
Suppose char$(\kappa(P))$ = char$(K)  = \ell$ and $a \in R_P$. Then $L = K[X]/(X^\ell -X + a)$ and
hence  $S_P/PS_P \simeq \kappa(P))[X]/(X^\ell - X + \overline{a})$. 
Thus, in either case, $S_P/PS_P$ is either a cyclic field extension of degree $\ell$ over $\kappa(P)$
 or the split extension of degree $\ell$ over $\kappa(P)$ and we denote these $S_P/PS_P$ by 
 $[a(P))$.    If $P = (\pi)$ for some 
 $\pi \in R$, then we also denote $[ a(P) ) $ by $[ a(\pi) )$.  
 If $P$ induces a discrete valuation $\nu$ on $K$, then we also denote $[a(P))$ by $[a(\nu))$.
  For an element $b \in R$, we also denote the image of $b$ in $R/P$ by $b(P)$. 
  If $b \in R$ and $c \in R/P$, we write $b = c \in R/P$ for $b \equiv c $ modulo $P$. 
  
 \begin{lemma}
 \label{choice_of_a_general} Let $A$ be a semi local regular ring of dimension at most two with field of fractions $F$. 
 Let $\ell$ be a prime not equal to the characteristic of $F$. Suppose that $F$ contains a primitive $\ell^{\rm th}$ root of 
 unity.   For each maximal ideal $m$ of
 $A$, let $[u_m)$ be  a cyclic extension of $A/m$ of degree $\ell$.  Then there exists   $a \in A$ such that  \\
 $\bullet$ $[a )$ is unramified on $A$ with residue field $[u_m)$ at each maximal ideal $m$ of $A$  \\
 $\bullet$ if $\ell  = 2$ and   $A/m$ is finite for all maximal ideals $m$ of $A$, 
 then  $a$ can be chosen to be  a sum of two squares in $A$.  
  \end{lemma}
 
 \begin{proof} Let $\rho \in F$ be a primitive $\ell^{\rm th}$ root of unity. 
  Let $m$ be a maximal ideal of $A$. If  char$(A/m) \neq \ell$, then 
 let $b_m = \frac{ 1 - u_m}{(\rho-1)^\ell} \in A/m$. If  char$(A/m) =  \ell$, then 
 let $b_m =   u_m  \in A/m$. Chose  $b \in A$ with $b = b_m \in A/m$ for all maximal ideals $m$ of $A$
 and $a = 1 - b(\rho- 1)^\ell$. Let $m$ be a maximal ideal of $A$.
 Suppose that char$(A/m) \neq \ell$.  Then, by the choice of $a$ and $b$, we have $a = 1 - b_m(\rho-1)^\ell = u_m \in A/m$.
 Thus $[a)$ is unramified at $m$ with the residue field $[u_m)$ at $m$.
 Suppose that char$(A/m) =\ell$.  Then, by  (\ref{epp}, \ref{epp_general}), $[a)$ is unramified at $m$ with the residue 
 field $[\overline{b})$. Since $b = b_m = u_m \in A/m$, the residue field of $[a)$ at $m$ is  $[u_m)$. 
 
 Suppose $\ell = 2$ and  $A/m$ is a finite field  for all maximal ideals $m$ of $A$. 
 Let $m$ be a maximal ideal of $A$. Suppose that char$(A/m) \neq 2$. Since every element of 
 $A/m$ is a sum of two  squares in $A/m$ (\cite[p. 39, 3.7]{Sc}), there exist $x_m, y_m \in A/m$ such that 
 $x_m^2 + y_m^2 = 1 - 4u_m$.  Suppose that char$(A/m) = 2$. Since $A/m$ is a finite field, 
 every element in $A/m$ is a square. Let $y_m \in A/m$ be such that $y_m^2 = u_m$. 
 Let $x, y \in A$ be such that for every maximal ideal $m$ of $A$, \\
 $\bullet$ if char$(A/m) \neq 2$, then $ x = \frac{x_m -1}{4} \in A/m$ and $y = \frac{y_m}{2} \in A/m$  \\
$\bullet$ if char$(A/m) = 2$, then $ x = 0 \in A/m$ and $y = y_m \in A/m$. \\
Let $ a = (1 + 4x)^2 + (2y)^2 \in A$.  Let $m$ be a  maximal ideal of $A$. 
Suppose char$(A/m)  \neq 2$. Then $ a = x_m^2 + y_m^2 = u_m \in A/m$
and hence $[a)$ is unramified at $m$ with residue field  at $m$ equal to $[u_m)$.
Suppose that char$(A/m) = 2$.  
Then $\frac{1 -a}{4} = u_m \in A/m$ and hence $[a)$ is unramified at $m$ with residue field 
$[u_m)$ (\ref{epp}, \ref{epp_general}). 

  \end{proof}

 \begin{lemma} 
 \label{localift}Let $R$ be a   regular ring of dimension at most 2 and $K$ its field of fractions.
 Let $\ell$ be a prime not equal to char$(K)$.  Suppose that 
 $K$ contains a primitive $\ell^{\rm th}$ root of unity $\rho$. 
 Let $L = K(\sqrt[\ell]{u})$ for some 
 $u \in R$. Let $m_1, \cdots , m_r, m_{r +1}, \cdots , m_n$ be   maximal ideals of $R$.
 Suppose that  $\kappa(m_j) = \ell$ and $L/K$ is unramified  at  $m_j$ for all $ r+1  \leq j \leq n$.
 Then there exists $v \in R$ such that  $L= K(\sqrt[\ell]{v})$, 
 $v \equiv u$ modulo $m_i$ for all $1 \leq i \leq r$ and $\frac{1 - v}{(\rho -1)^\ell}$ is a unit at $m_j$ for all $r+1 \leq j \leq n$.
  \end{lemma}
  
  \begin{proof} For a maximal ideal $m$ of $R$, let $K_m$ denote the field of factions of the completion 
  of $R$ at $m$.
   
  Let $r+1 \leq j \leq n$. Since char$(\kappa(m_j)) = \ell$ and 
  $L/K$ unramified at $m_j$, the residue field of $L$ at $m_j$ is $\kappa(m_j)[X]/(X^\ell - X + \overline{w_j})$ for some
  $w_j \in R_{m_j}$.  Since the residue field of $K[X]/(X^\ell - (1 - w_j(\rho - 1)^\ell))$ is 
  isomorphic to $\kappa(m_j)[X]/(X^\ell - X + \overline{w_j})$ (\ref{epp}, \ref{epp_general}), $L \otimes K_{m_j} 
  \simeq K_{m_j}[X]/(X^\ell - (1 - w_j(\rho - 1)^\ell))$. Since char$(K) \neq \ell$ and $L = K(\sqrt[\ell]{u})$, 
  there exists $\theta_j \in K_{m_j}$ such that $u\theta_j^\ell = 1 - w_j( \rho - 1)^\ell$. 
  Let $\theta \in R$ be such that $\theta \equiv 1 $ modulo $m_i$ for $1 \leq i \leq r$
  and $\theta \equiv w_j$ modulo $m_j$ for $r+1 \leq j \leq n$.
  Then $v = u \theta^\ell$ has the required properties.  
  \end{proof}
 
 The following is a generalization of a result of Saltman (\cite[Proposition 0.3]{S3}). 
 \begin{lemma} 
 \label{saltman_exact_seq}
 Let $A$ be a UFD. For $1 \leq i \leq n$,  let  $I_i = (a_i) \subset A$ with gcd$(a_i, a_j)   = 1$ for all $i \neq j$. 
For each $i < j$, let $I_{i j} = I_i +  I_j$. Suppose that the ideals $I_{i, j}$ are
comaximal. Then
$$
A \to \bigoplus_i A/I_i  \to \bigoplus_{i < j} A/I_{ij}
$$
 is exact, where for $i < j$, the map from $A/I_i \oplus A/I_j \to A/I_{ij}$ is 
 given by $(x, y) \mapsto x - y$.
 \end{lemma}
 
 \begin{proof} Proof by induction on $n$. The case  $n = 2$ is in  (\cite[Lemma  0.2]{S3}).
 Assume that $n \geq 3$. Let $a_i \in A/I_i$ maps to zero in $\oplus A/ I_{ij}$.
 By induction, there exists $b \in A$ such that $b = a_i \in A/I_i$ for $1 \leq i \leq n-1$.
 We claim that  $I_1 \cap \cdots \cap I_{n-1} + I_n = (I_1 + I_n) \cap \cdots \cap (I_{n-1} + I_n)$.
 Since  both sides contain $I_n$,  it is enough to prove  the equality   after 
 going modulo $I_n$. Since gcd$(a_i, a_j) = 1$ for all $i \neq j$, we have $I_1 \cap \cdots \cap I_{n-1} 
 = Aa_1\cdots a_{n-1}$ and hence $I_1 \cap \cdots  \cap I_{n-1} + I_n / I_n = (A/I_n)\overline{a}_1 \cdots 
 \overline{a}_{n-1}$. Since $I_{ij}$ are comaximal, $I_{in} /I_n = (A/I_n)\overline{a}_i$ are comaximal for 
 $1 \leq i \leq n-1$ and hence $(A/I_n)\overline{a}_1 \cdots \overline{a}_{n-1} = (A/I_n)\overline{a}_1 \cap 
 \cdots \cap (A/I_n)\overline{a}_{n-1}$. Let $b_1 \in A/(I_1 \cap \cdots \cap I_{n-1} )$ be the image of $b$. 
 Then, by the case $n = 2$,  there exists $a \in A$ such that $a = b_1 \in  A/I_1 \cap \cdots \cap I_{n-1} $ 
 and $ a =  a_n \in A/I_n$.Thus $a$ has the required properties. 
 \end{proof}

 \section{Central simple algebras}
 \label{csa}
   
 Let $K$ be a field, $L/K$ a cyclic extension of degree $n$ with $\sigma \in $ Gal$(L/K)$ a generator and $b \in  K^*$. 
 Let $(L, \sigma, b)$ denote the cyclic algebra $L \oplus Lx \oplus \cdots \oplus Lx^{n-1}$
 with relations $x^n = b$, $x \lambda  = \sigma(\lambda) x$ for all $\lambda\in L$.  
 Then $(L, \sigma, b)$ is  a central simple 
 algebra over $K$ and represents an element in the $n$-torsion subgroup  $\Brn(K)$ of 
 the Brauer group $\Br(K)$ (\cite[Theorem 18, p. 98]{Albert}). 
  Suppose that $n$ is coprime to char$(K)$ and $K$ contains  a primitive $n^{\rm th}$ root
 of unity. Then  $L = K(\sqrt[n]{a})$ for some $a \in K^*$.
 Fix a primitive  $n^{\rm th}$ root of unity $\rho$ in $K$.  Let $\sigma$ be the generator of Gal$(L/K)$ given by 
 $\sigma(\sqrt[n]{a})  = \rho \sqrt[n]{a}$.  Then, the cyclic algebra $(L, \sigma, b)$ is denoted by 
 $[a, b)$. 
 Suppose that $n$ is prime and equal to   char$(K)$. Then,   $L = K[X]/(X^n - X + a)$
 for some $a \in K$. 
   If $\sigma$ is the generator of Gal$(L/K)$  given by $\sigma(x) = x +1$, 
 then the cyclic algebra $(L, \sigma, b)$ is also denoted by $[a, b)$.

 For any   Galois module $M$  over $K$, let $H^n(K, M)$ denote the 
Galois cohomology of   $K$ with coefficients in $M$. 
 Let $K$ be a field and $\ell$ a prime. 
 Let $\Z/\ell(i)$ be the Galois modules over $K$  as in (\cite[\S 0]{K2}).
 We have a canonical isomorphisms $H^1(K, \Z/\ell) \simeq {\rm Hom}_{\rm cont}({\rm Gal}(K^{\rm ab}/K), \Z/\ell)$
 and $\Brl(K) \simeq H^2(K, \Z/\ell(1))$, where $K^{\rm ab}$ is the maximal abelian extension of $K$ (\cite[\S 0]{K2}). 

 Suppose $A$ is a regular domain   with field of fractions $F$. We say that an element $\alpha \in H^2(F, \Z/\ell(1))$
 is {\it unramified} on $A$ if $\alpha$ is represented by a central simple algebra over $F$ which comes from 
 an Azumaya algebra over $A$. If it is not unramified, then we say that $\alpha$ is {\it ramified} on $A$.
 Suppose $P$ is a prime ideal of $A$ and $\alpha \in H^2(F, \Z/\ell(1))$. We say that $\alpha$ is {\it unramified}
 at $P$ if $\alpha$ is unramified at $A_P$.  If $\alpha$ is not unramified at $P$, then 
 we say that $\alpha$ is {\it ramified} at $P$. 
  Suppose that $\alpha$ is unramified at $P$. 
 Let $\AA$ be an Azumaya algebra over $A_P$ with the class of $\AA \otimes_{A_P} F$ equal to $\alpha$.
The algebra $ \overline{\alpha} = \AA \otimes_{A_P} (A_P/PA_P)$ is called the {\it specialization} of $\alpha$ at $P$.
 Since $A_P$ is  a regular local ring, the class of $\overline{\alpha}$ is independent of the choice of 
 $\AA$. Let $a, b \in F$ and $\alpha = [a, b) \in H^2(F,\Z/\ell(1))$. If the cyclic extension $[a)$ is unramified at $P$
  and $b$ is a unit at $P$,  then $\alpha$ is unramified at $P$ and  the specialization of $\alpha$ at $P$ is $[a(P),  b(P))$,
   where $[a(P))$ is the residue field of $[a)$ at $P$ and $ b(P)$ is the image of $b$ in $A_P/PA_P$. 
 
 Suppose that $R$ is a discrete valuation ring with field of fractions $K$ and residue field $\kappa$.
 Let $\ell$ be a prime not equal to char$(K)$. Suppose that char$(\kappa) \neq \ell$  or char$(\kappa) = \ell$ with
  $\kappa =  \kappa^\ell$. 
  Then there is a {\it residue homomorphism } $\partial : H^2(K, \Z/\ell(1)) \to H^1(\kappa, \Z/\ell)$ (\cite[\S 1]{K2}).
 Further a class $\alpha \in H^2(K, \Z/\ell(1))$ is unramified at $R$ if and only if $\partial(\alpha) = 0$.
 Let $a, b \in K^*$. If $[a)$ is unramified at $R$, then $\partial([a, b)) = [a(\nu))^{\nu(b)}$, where $\nu$ is the 
 discrete valuation on $K$. In particular if $[a)$ is unramified on $R$ and $\ell$ divides $\nu(b)$, then 
 $[a, b)$ is unramified on $R$.

 \begin{lemma}
\label{2dim_unramified} (cf. \cite[Lemma 3.1]{LPS})
Let $A$ be a regular  ring of dimension 2 and $F$ its field of fractions. 
Let $\ell$ be a   prime not equal to char$(F)$ and  $\alpha \in H^2(F, \Z/\ell(1))$.
If $\alpha$ is unramified at all height one prime ideals of $A$, then 
$\alpha$ is unramified on $A$. 
 \end{lemma}
 
 \begin{lemma}
 \label{ind_ell}
 Let $R$ be a complete discrete valued field with field of fractions $K$ and residue field 
 $\kappa$. Let $\ell$ be a prime not equal to char$(\kappa)$. Let $D$ be a central simple algebra of index $\ell$ over $K$.
Suppose that $D$ is ramified at $R$.
If  $L/K$ is  the  unramified extension of $K$ with residue field equal to the residue of $D$ at $R$, 
then $D \otimes L$ is a split algebra.  
 \end{lemma}
 
 \begin{proof} We have $ D= D_0 \otimes (L, \sigma, \pi)$ for some generator of Gal$(L/K)$, $\pi$ a
 parameter in $R$ and $D_0$ unramified at $R$ (cf. \cite[Lemma 4.1]{PPS}). 
 Further  $\ell = $ ind$(D) = $ ind$(D_0 \otimes L)[L : K]$ (cf. \cite[Lemma 4.2]{PPS}).
 Since $D$ is ramified at $R$, $[L : K] = \ell$ and hence $D_0 \otimes L = 0$.
 Hence $D_0 = (L, \sigma, u)$ for some $u \in K$ and $D = (L, \sigma, u\pi)$.
 Thus $D \otimes L$ is a split algebra.
 \end{proof}

  \begin{lemma}
 \label{curve_point_23} Let $A$ be a  complete regular local ring of dimension 2 
 with field of fractions $F$ and residue field  $\kappa$. Suppose that $\kappa$ is  a finite field.
  Let $m = (\pi, \delta)$ be the maximal ideal of $A$.
  Let $\ell$ be a prime not equal to char$(F)$ and  $\alpha  = [ a,   b ) \in H^2(F, \Z/\ell(1))$ for some $a, b \in F^*$.
  Suppose that \\
  $\bullet$ if char$(\kappa)= \ell$, then the  cyclic extension $[a )$ is  unramified on  $A$ \\
  $\bullet$  $\alpha$ is unramified on $A$ except possibly at $\delta$ \\
$\bullet$ the specialization of $\alpha$ at $\pi$ is unramified on $A/(\pi)$. \\
Then   $\alpha = 0 $.
\end{lemma}
 
 \begin{proof}  Suppose that char$(\kappa) \neq \ell$.
 Then, it follows from  (\cite[Proposition 3.4]{RS}) that $\alpha = 0$ (cf. \cite[Corollary 5.5]{PPS}).

 Suppose that char$(\kappa) = \ell$. 
 Since $F$ is the field of fractions of $A$, without loss of generality, we assume that 
$b \in A$ and not divisible by $\theta^\ell$ for any prime $\theta \in A$. 
Write $b = v\delta^{n} \theta_1^{n_1} \cdots \theta_r^{n_r}$ for some distinct primes $\theta_i \in A$ with 
$(\delta) \neq (\theta_i)$ for all $i$,  $1 \leq   n_i \leq \ell -1$, $0 \leq n \leq \ell-1$ and $v \in A$ a unit.
Since $\kappa$ is a finite field,  $A$ is complete and $[a)$ is unramified on $A$, we have 
$[a, v) =0$ and hence $\alpha = [a, b) = [a, \delta^n \theta_1^{n_1} \cdots \theta_r^{n_r})$.

 Since $[a)$ is unramified on $A$, for any prime $\theta \in A$,   $[a, \theta)$ is unramified on $A$
 except possibly at $\theta$. 
 Let $1 \leq j \leq r$.   Since $\alpha = [a,b) = [a, \delta^n)  \prod [a, \theta_i^{n_i})$, 
 $[a, \delta^n)$ and $[a,\theta_i^{n_i})$ are   unramified at $\theta_j$ for all $i \neq j$, $[a, \theta_j^{n_j})$
 is unramified at $\theta_j$ and hence $[a, \theta_j^{n_j})$ is unramified on $A$ (cf. \ref{2dim_unramified}). 
Since $\kappa$ is a finite field and $A$ is complete, $[ a , \theta_j^{n_j} )  = 0$. 
Thus, we have $\alpha   = [ a, \delta^n )$. 

If  $n = 0$, then $\alpha   = 0$.
Suppose $1 \leq n \leq \ell-1$.  
 Let $\overline{\alpha}$ be the specialization of $\alpha$ at $\pi$.
Since $\alpha = [a, \delta^n)$ and 
$[a)$ is unramified at $\pi$, 
 we have $\overline{\alpha}  = [ a(\pi), \overline{\delta}^n )$,
where $[ a(\pi) )$ is the residue field of $[a )$ at $\pi$ and $\overline{\delta}$ is the image of
$\delta$ in $A_P/(\pi)$.  Since   $\overline{\alpha}$ is unramified on $A/(\pi)$,
$A$ is complete  and $\kappa$ 
is a finite field,   $\overline{\alpha}  = [a(\pi), \overline{\delta}^n)   = 0$.
Since $\partial(\overline{\alpha}) = [a(m))^n = 1$ and $n$ is coprime to $\ell$,  
$[a(m)) = 0$. Since $A$ is complete,    $[a )$ is trivial   and hence $\alpha = 0$.
\end{proof}

We now  recall the  chilly points and  the chilly loops associated to a central simple algebra, due to Saltman
(\cite{S2}, \cite{S3}). Let $\XX$ be a regular integral excellent scheme of dimension 2 and $F$ its field of fractions.
Let $\ell$ be a prime  which is not equal to char$(F)$.  Suppose that $F$ contains a primitive $\ell^{\rm th}$ root of unity.
Let $\alpha \in H^2(F, \Z/\ell(1))$.  Suppose that ram$_{\XX}(\alpha) \subset \{ D_1, \cdots ,D_n\}$  for 
 some regular irreducible curves $D_i$ on $\XX$ with normal crossings. 
Let $P \in D_i \cap D_j$ be a closed point.  If char$(\kappa(P)) \neq \ell$, then, we say that $P$ is a {\it chilly point} of 
$\alpha$ if $ \alpha = \alpha_0 + (u, \pi_i \pi_j^s)$ for some $\alpha_0$ unramified at $P$, $u$ a unit at $P$
and $\pi_i, \pi_j$ primes at $P$ defining $D_i$ and $D_j$ at $P$ respectively. 

Let $\Gamma$ be a  graph with vertices $D_i$'s and edges as chilly points,i.e. 
two distinct  vertices $D_i$ and $D_j$  have an edge between them if there is a chilly point in $D_i \cap D_j$. 
A loop in this graph is called a {\it chilly loop} on $\XX$. 
Let $\XX[\frac{1}{\ell}]$ be the open subscheme 
of $\XX$ obtained by inverting $\ell$. 
Since, by the definition of chilly point, char$(\kappa(P)) \neq \ell$ for any 
chilly point $P$, we have the following  

\begin{prop} 
\label{chillyloops}  (\cite[Corollary 2.9]{S2})
There exists a sequence of blow-ups $\XX' \to \XX$ centered at closed points $P \in \XX[\frac{1}{\ell}]$ such that 
$\alpha$ has no chilly loops on $\XX'$.  
\end{prop}

Let $K$ be a global field and $\ell$ a prime. Let $\beta \in \Brl(K)$.
 Let $\nu$ be a discrete vacation of $K$,  $K_\nu$ the completion of $K$ at $\nu$ and $\kappa(\nu)$ the residue field at $\nu$. 
  Since $K_\nu$ is a local field, 
 the invariant map gives an isomorphism $\partial_\nu : \Brl(K_\nu) = H^2(K_\nu, \Z/\ell(1))
  \to H^1(\kappa(\nu), \Z/\ell)$. 
 
 \begin{prop}
 \label{global_field} Let $K$ be a global field   and $\ell$ a prime.  If $\ell$ 
 is not equal to  char$(K)$, then assume that $K$ contains a primitive $\ell^{\rm th}$ root of unity $\rho$.
 Let  $\beta \in  \Brl(K)$.   Let $S$ be a finite set of discrete valuations of $K$ containing 
 all the discrete valuations $\nu$ of $K$ with $\partial_\nu(\beta) \neq 0$.  Let $S'$ be a finite set of 
 discrete valuations of $K$ with $S \cap S' = \emptyset$.
 Let $a \in K^*$ and for each $\nu \in S'$,  let    $n_\nu \geq 2$ be  an  integer. 
 Suppose that   for every $\nu \in S$,  $[a )$   is unramified at   $\nu$ with 
 $\partial_\nu(\beta )  = [a(\nu))$.  
Further assume that  if $\ell = 2$, then   $\beta \otimes K_\nu(\sqrt{a}) = 0$ for all real places $\nu$ of $K$. 
 Then there exists $b \in K^*$  
such that \\
$\bullet$   $\beta = [ a, b ) $  \\
 $\bullet$  if $\nu \in S$,   then $\nu(b) = 1$   \\
$\bullet$  if $\nu \in S'$, then  $\nu(b-1) \geq n_\nu$.
 \end{prop}
 
 \begin{proof} Let $L = [a)$. 
 Let $\nu \in S$.  If   $\partial_\nu(\beta) = 0$, then $\beta \otimes K_\nu = 0$ (\cite[p. 131]{CFANT}).
 Suppose that $\partial_\nu(\beta) \neq 0$. Then $[a(\nu))$ is a  field extension of $\kappa(\nu)$ of degree $\ell$
  and hence $L \otimes_K K_\nu$ is a degree $\ell$ field extension of $K_\nu$.
  Thus   $ \beta \otimes_K (L \otimes K_\nu) = 0$ (cf. \cite[p. 131]{CFANT}).  
 Suppose $\nu$ is a real place of $K$. Then, by the assumption on $a$, 
  $\beta \otimes _K( L \otimes_K K_\nu)  = 0$. 
 Thus $\beta \otimes L = 0$ (cf. \cite[p. 187]{CFANT})   
 and hence there exists $c \in K^*$ such that $\beta = [ a ,   c )$ (cf. \cite[p. 94]{Albert}). 

Let $R$ be the semi local ring at 
 the  discrete valuations in $S \cup  S'$.  
 Replacing $c$ by $c\theta^{n\ell}$ for some $\theta \in K^*$ and $n \geq 1$,
 we assume that  $c \in R$. 
 For  $\nu \in S \cup S'$, let  $\pi_\nu \in R$ be a parameter at $\nu$.  
 Let $\nu \in S$.
 Since $[a)$ is unramified at $\nu$,   $\partial_\nu(\beta) = \partial_\nu([a, c)) = [a(\nu))^{\nu(c)}$. 
 Suppose   $[a(\nu))$ is   non-trivial. Since, by the hypothesis,  $\partial_\nu(\beta) =  [a(\nu) )$, 
  $\nu(c) -1$ is divisible by $\ell$.
  Since $[L: K] = \ell$, $\pi_\nu^{\nu(c) -1}  $ is a norm from    $L\otimes_K K_\nu/K_\nu$. 
 Suppose that   $[a(\nu))$ is   trivial.   
 Then   $L \otimes_KK_\nu$ is the split extension and hence  every element of  $K_\nu$
 is a norm from $L \otimes _KK_\nu/K_\nu$. Thus for each $\nu \in S$, 
 there exists   $x_\nu \in L \otimes_K K_{\nu}$ with norm $ \pi_\nu^{\nu(c) - 1}$. 
 Let $\nu \in S'$. Then $\partial_\nu(\beta) = 0$  and we have 
$\beta \otimes K_{\nu} = [a, c ) \otimes K_\nu = 0$ (\cite[p. 131]{CFANT}).
Hence  $c$ is a norm from $L \otimes_K K_{\nu}$. 
For each $\nu \in S'$, $x_\nu \in L \otimes_K K_{\nu}$ with norm   $c$.  
 Let $z \in L$  be  sufficiently close to $x_\nu$ 
 such that $\nu(N_{L\otimes_KK_\nu}(z) - \pi_\nu^{\nu(c) -1}) \geq \nu(c) $ for 
 all $\nu \in S$ and $\nu(N_{L\otimes_KK_\nu}(z) - c) \geq \nu(c) + n_\nu $ for 
 all $\nu \in S'$.  
 
 Let  $d$ be the norm of $z$ and $b = cd^{-1}$.  Then $\beta =  [ a,  cd^{-1} ) = [a, b)$.  
 Let $\nu \in S$.  Since $\nu(d  - \pi_\nu^{\nu(c) -1}) \geq \nu(c) $,
 we have $ \nu(d) =   \nu(c) - 1$ and hence  $\nu(b)= \nu(cd^{-1}) = 1$.
 Let $\nu \in S'$.  Since   $\nu(d  - c) \geq \nu(c) + n_\nu  \geq 2$,   $\nu(d) = \nu(c)$ and 
 $\nu(b - 1) = \nu(cd^{-1} - 1)  \geq n_\nu$.
  \end{proof}

 \section{A complex of Kato}
\label{kato} 

Let $K$ be a complete discrete valued field with residue field $\kappa$. Let $\ell$ be a prime not equal to 
characteristic of $K$. If $\ell  =$ char$(\kappa)$, then assume that $[\kappa : \kappa^\ell] \leq \ell$.
Then, there is a residue homomorphism $\partial : H^3(K, \Z/\ell(2)) \to H^2(\kappa, \Z/\ell(1))$ (\cite[\S 1]{K2}).
We say that an element  $\zeta \in H^3(K, \Z/\ell(2))$ is {\it unramified } at the discrete valuation of $F$ if $\partial(\zeta) = 0$.

Let    $\XX$ be    a   two-dimensional  regular integral excellent scheme 
 and $F$  the function field of $\XX$.  For $x \in \XX$,  let $F_x$ be the field of fractions of the completion 
 $\hat{A}_x$ of the local ring $A_x$ at $x$ on $\XX$ and  $\kappa(x)$  the 
 residue field at $x$.  Let $x \in \XX$ and $C$ be the closure of $\{ x\}$ in $\XX$.
 Then, we also denote $F_x$ by $F_C$. If the dimension of $C$ is one, then $C$ defines a discrete 
 valuation $\nu_C$ (or $\nu_x$) on $F$.
 Let $\XX_{(i)}$ be the set of points of $\XX$ with the dimension of the closure of $\{x \}$ equal to  $i$.
  Let $\ell$ be a prime not equal to char$(F)$. Suppose that $F$ contains a primitive $\ell^{\rm th}$ root of unity.
 If $P \in \XX_{(0)}$ is a  closed   point   of $\XX$  with char$(\kappa(P)) = \ell$,
  then we assume  $\kappa(P) =  \kappa(P)^\ell$. 
 Let $x \in \XX_{(1)}$.  We have a   {\it residue homomorphism}   $\partial_x : H^3(F, \Z/\ell(2) ) \to H^2(\kappa(x), 
  \Z/\ell(1))$ (\cite[\S 1]{K2}).  We say that an element $\zeta \in H^3(F, \Z/\ell(2))$ is  {\it unramified} at $x$ (or $C$) 
  if $\zeta$ is unramified at $\nu_x$.  
  Further if  $P \in \XX_{(0)}$ is   in the closure of $\{ x \}$,
 then  we have a  {\it residue  homomorphism } $\partial_P : H^2(\kappa(x), \Z/\ell(1)) \to H^1(\kappa(P), 
 \Z/\ell)$ (\cite[\S 1]{K2}).  For $x \in \XX_{(1)}$, if $C$ is the closure of $\{ x \}$, we also denote 
 $\partial_x$ by $\partial_{C}$. An element $\alpha \in H^2(\kappa(x), \Z/\ell(1)) \simeq \Brl(\kappa(x))$ is unramified at 
 $P$ if and only if $\partial_P(\alpha) = 0$. 
We use the additive notation for the group operations on  $H^2(F, \Z/\ell(1))$ and $H^3(F, \Z/\ell(2))$
and multiplicative notation for the group operation on $H^1(F, \Z/\ell)$.  
 
\begin{prop}
\label{complex} (\cite[Proposition 1.7]{K2})  Then  
$$
H^3(F,  \Z/\ell(2))  \buildrel{\partial}\over{\to} \oplus_{x \in \XX_{(1)}}H^2(\kappa(x),  \Z/ \ell(1)) 
\buildrel{\partial}\over{\to} \oplus_{P\in \XX_{(0)}} H^1(\kappa(P), \Z/\ell).
$$
is a complex, where the maps are given by the residue homomorphism. 
\end{prop}

   \begin{lemma} 
\label{residue} (\cite[\S 3.2, Lemma 3]{K}, \cite[Lemma 1.4(3)]{K2}) 
Let $x \in \XX_{(1)}$ and $\nu_x$ be the discrete valuation on $F$ at $x$. 
Then $\partial_x : H^3(F_x, \Z/\ell(2)) \to H^2(\kappa(x), \Z/\ell(1))$ is an isomorphism.
Further if  $\alpha \in H^2(F, \Z/\ell(1))$ is  unramified at $x$  and $f \in F^*$,  
then $\partial_x(\alpha \cdot (f)) = \overline{\alpha}^{\nu_x(f)}$.
\end{lemma}
 
 The following is a consequence of (\ref{complex}). 
 
 \begin{cor}
 \label{complex_P}
  Let $C_1$ and $C_2$  be two irreducible regular curves in $\XX$
 interesting transversely at a closed point $P$. Let $\zeta \in H^3(F,\Z/\ell(2))$.
 Suppose that $\zeta$ is unramified at  all codimension one points of $\XX$ passing 
 through $P$ except possibly at $C_1$ and $C_2$. Then 
 $$\partial_P(\partial_{C_1}(\zeta)) =  \partial_P(\partial_{C_2}(\zeta))^{-1}.$$ 
 \end{cor}
 
 
%

 \begin{cor}
\label{curve-point} Let   $C$ be an irreducible curve on $\XX$ and  $P \in C$
  with $C$ regular  at $P$.  Let $\zeta \in H^3(F,\Z/\ell(2))$.
 Suppose that $\zeta$ is unramified at  all codimension one points of $\XX$ passing 
 through $P$ except possibly at $C$. If $\kappa(P)$
is finite, then $\zeta \otimes F_P = 0$.   
In particular if $\kappa(P)$ is finite,  then $\zeta$ is unramified at every discrete valuation of $F$ centered at $P$. 
\end{cor}

 \begin{proof} Since $C$ is regular at $P$, there exists an irreducible curve $C'$  passing through $P$ and 
 intersecting $C$ transversely at $P$. Then, by (\ref{complex_P}), we have $\partial_P(\partial_{C}(\zeta)) = \partial_P(\partial_{C'}(\zeta))^{\ell-1}$.  Since, by  assumption, $\partial_{C'}(\zeta) = 0$, we have $\partial_P(\partial_{C}(\zeta)) = 1$. 
 
 Let $\pi \in A_P$ be a prime defining $C$ at $P$. Since $C$ is regular at $P$, $A_P/(\pi)$ is a
   discrete valued ring  with residue field $\kappa(P)$ and $\kappa(C)$ is the field 
 of fractions of $A_P/(\pi)$.  Further $\pi$ remains a regular 
 prime in $\hat{A}_P$ and $\hat{A}_P/(\pi)$ is the completion of $A_P/(\pi)$. 
 In particular the field of fractions of $\hat{A}_P/(\pi)$ is the completion  $\kappa(C)_P$ of the field 
 $\kappa(C)$ at the discrete valuation  given by the discrete  valuation ring $A_P/(\pi)$. 
 Let $\tilde{\nu}$ be the  discrete valuation on $F_P$ given by the  height one prime ideal $(\pi)$
 of $\hat{A}$ and $\nu$  the discrete valuation of $F$ given by the  height one prime ideal $(\pi)$
 of $A$. Then the restriction of $\tilde{\nu}$ to $F$ is $\nu$ and the residue field $\kappa(\tilde{\nu})$ at $\tilde{\nu}$
 is  $\kappa(C)_P$. 
 
 Since $\partial_P(\partial_{C}(\zeta)) = 1$,  we have $\partial_C(\zeta) \otimes \kappa(C)_P = 0$ ( \cite[Lemma 1.4(3)]{K2}).
 Since $\partial_{\tilde{\nu}}(\zeta \otimes F_P) = \partial_C(\zeta) \otimes \kappa(C)_P = 0$. 
 Let $F_{P, \tilde{\nu}}$ be the completion of $F_P$ at $\tilde{\nu}$. 
 Since $\partial_{\tilde{\nu}} : H^3(F_{P, \tilde{\nu}}, \Z/\ell(2)) \to H^2(\kappa(C)_P, \Z/\ell(2))$ 
 is an isomorphism (\cite[Lemma 1.4(3)]{K2}), $\zeta \otimes F_{P, \tilde{\nu}} = 0$.
 
 Let $\nu'$ be a discrete valuation of $F_P$ given by a hight one prime ideal of 
 $\hat{A}$ not equal to $(\pi)$.
 Then, by the assumption on $\zeta$, $\partial_{\nu'}(\zeta \otimes F_P) = 0$
 and hence $\zeta \otimes F_{P, \nu'} = 0$ (\cite[Lemma 1.4(3)]{K2}), 
 where $F_{P, \nu'}$ is the completion of $F_P$ at $\nu'$. 
 Hence, by (\cite[Lemma 6.2]{Saito}), $\zeta \otimes F_P = 0$. 
  \end{proof}
 
\section{A local global principle} 
Let $\XX$, $F$ and $\ell$ be as in \S \ref{kato}. Let $\zeta \in H^3(F, \Z/\ell(1))$. 
Let $\alpha = [a, b)  \in H^2(F, \Z/\ell(1))$.  In this section we show that under some additional assumptions 
on $\XX$, $\zeta$ and $\alpha$, there exists $f \in F^*$ such that $  \partial_x(\zeta - \alpha \cdot (f))$ is unramified 
at all the discrete valuations of $\kappa(x)$ centered at closed points of   $\overline{\{ x \}}$ for all $x \in \XX_{(1)}$ (see \ref{local_global}). 

Let $\zeta \in H^3(F, \Z/\ell(1))$ and $\alpha = [a, b)  \in H^2(F, \Z/\ell(1))$.
For the rest of this section, we assume the following. 
 
\begin{assumptions}
\label{assumptions}
Suppose  ($\XX$, $\zeta $, $\alpha$) satisfy the following conditions. 
\begin{enumerate}
\item[{A1)}] ram$_\XX(\zeta)  =  \{ C_1, \cdots , C_r \}$, $C_i$'s   are regular irreducible curves with normal crossings  \\
\item[{A2)}]   ram$_\XX(\alpha) = \{ D_1, \cdots ,  \cdots , D_n \}$,  $D_j$'s  are  regular curves with normal crossings
and  $D_i \neq C_j$ for all $i, j$   \\
\end{enumerate}
By reindexing, we have ram$_\XX(\alpha) = \{ D_1, \cdots , D_m, \cdots , D_n \}$,  
with char$(\kappa(D_i)) = \ell$  for $1 \leq i \leq m$ and char$(\kappa(D_j)) \neq \ell$ for 
$m+1 \leq j \leq n$. \\

 \begin{enumerate}
 \item[{A3)}] and  $D_i \cap D_j = \emptyset$ for all $ 1 \leq i \leq m$ and
 $m+1 \leq j \leq n$ \\
 \item[{A4)}] if $P \in D_i \cap D_j$ for some $m+1 \leq i < j \leq n$, then char$(\kappa(P)) \neq \ell$ \\
 \item[{A5)}]  there are no chilly loops (see \S \ref{csa}) for $\alpha$ on $\XX$ \\
 \item[{A6)}] $\partial_{C_i}(\zeta)$  is the specialization of $\alpha$ at $C_i$ for all $i$ \\ 
\item[{A7)}] $C_i \cap D_j = \emptyset$ for all $i$  and $1 \leq j \leq m$ \\
\item[{A8)}] if $P \in C_i \cap D_s$ for some $i$ and $s$, then $P \in C_i \cap C_j$ 
for some $i \neq j$ \\
\item[{A9)}]   for every $i \neq j$, there is at most one $D_t$  such that 
$C_i \cap C_j \cap D_t \neq \emptyset$ \\
\item[{A10)}] if $P \in \XX_{(2)}$ with char$(\kappa(P)) = \ell$ and $P \in D_i$ for some $i$,  
then   $\frac{1-a}{(\rho-1)^\ell} \in A_P$. \\
\item[{A11)}] if $P \in C_i \cap C_j \cap D_t$ for some $i < j$ and for some $t$, 
then $D_t$ is given by a regular  prime $u\pi_i^{\ell-1} + v\pi_j$  at $P$, for some 
prime $\pi_i$ (resp.  $\pi_j$) defining $C_i$ (resp.  $C_j$) at $P$ and  units $u, v$ at $P$  \\

\end{enumerate} 
\end{assumptions}

Let $\PP$ be a finite set of closed points of $\XX$ containing $C_i \cap C_j$, $D_i \cap D_j$ for all $i \neq j$,
$C_i \cap D_j$ for all $i, j$ 
and at least one point from each $C_i$ and $D_j$. Let $A$ be the regular semi local ring at $\PP$ on $\XX$.
For every $P \in \PP$, let   $M_P$ be the maximal ideal of $A$ at $P$.  
For  $1 \leq i \leq r$ and $1 \leq j \leq n$, let $\pi_i \in A$ be a prime defining $C_i$ on $A$ and $\delta_j \in A$ 
a prime defining $D_j$ on $A$.

   
 Since char$(F) \neq \ell$, $[a)$ is the extension $F(\sqrt[\ell]{a})$. Since 
$\alpha = [a, b)$, without loss of generality, we assume that $a, b \in A$ and both $a$ and $b$
are $\ell^{\rm th}$ power free.

\begin{lemma} 
\label{square} 
For $1 \leq j \leq n$, let  $n_j = \ell \nu_{D_j}(\ell)  + 1$.
Then there exists a unit $u \in A$ such that 
 $ u\prod \pi_i$ is an $\ell^{\rm th}$ power   modulo $\delta_j^{n_j}$ for all $1 \leq j \leq n$.
 In particular $u\prod \pi_i \in F_{D_j}^\ell$ for all $j$.
\end{lemma}
 
 \begin{proof}Let $\pi = \prod_1^r \pi_i$ and $\delta = \prod_1^m \delta_j^{n_j}$.
Since, by the assumption A7), $C_i \cap D_j  = \emptyset$ for all $i$ and $1 \leq j \leq m$, 
 the ideals $A\pi$ and $A\delta$ are comaximal  in $A$.
In particular the image of $\pi$ in $A/(\delta)$ is a unit.  By the Chinese remainder theorem, there exists 
   $u_1 \in A$ be 
such that $u_1 = \pi \in A/(\delta)$, $u_1 = 1 \in A/(\pi)$ and
$u_1 = 1 \in A/M_P$ for all $P \in\PP \setminus ((\cup_1^r C_i) \cup (\cup_1^m D_j))$. 
Since the image of $\pi$ in $A/(\delta)$ is a unit,   $u_1$ is a unit in $A$.  Let $\pi' = u_1^{-1}\pi$.  
 
Let $m+1  \leq s \leq n$ and  $a_s $ be the image of $\pi'$  in $ A/(\delta_s)$.   We claim that 
$a_s = w_sb_s^\ell$ for some  $w_s, b_s \in A/(\delta_s)$ with $w_s$  a unit in $A/(\delta_s)$. 
Let $M$ be a maximal ideal of $A/(\delta_s)$. 
Then $M =  M_P/(\delta_s)$ for some $P \in D_s \cap \PP$. 
Suppose $P \not\in C_i$ for all $i$. Then $\pi'  $ is a unit at $P$ and 
hence  $a_s $ is a unit at $M$. Suppose $P \in C_i$ for some $i$.
Then $P \in C_i \cap D_s$. Thus, by the assumption A8), there exists $j \neq i$ such that 
$P \in C_i \cap C_j$.  
Suppose $i < j$. Then, by the assumption A11),    $\delta_s = v_i\pi_i^{\ell-1} + v_j \pi_j$ for some 
units $v_i$ and $v_j$ at $P$. Hence 
$$a_s \equiv  u_1^{-1}( \prod_{t \neq i, j} \pi_t ) \pi_i\pi_j =
u_1^{-1}( \prod_{t \neq i, j} \pi_t ) \pi_i(-\frac{v_i}{v_j} \pi_i^{\ell-1}) = 
u_1^{-1}( \prod_{t \neq i, j} \pi_t )(-\frac{v_i}{v_j}) \pi_i^\ell {~~\textrm{modulo}~} \delta_s.$$
Since $\pi_t$,  $t\neq i, j$, is a unit at $P$ (assumption A1)), $a_s \equiv w_P \pi_j^{\ell}$  modulo $\delta_s$, 
for some
$w_P \in A/(\delta_s)$ a unit at $P$. Suppose $i > j$. Then $\delta_s = v_j\pi_j + v_i \pi_i^{\ell-1}$ for some 
units $v_i$ and $v_j$ at $P$. Hence, as above,  $a_s \equiv w_P\pi_i^{\ell}$ modulo $\delta_s$, for some
$w_P \in A/(\delta_s)$ a unit at $P$. Hence at every maximal ideal of $A/(\delta_s)$, 
$a_s$ is a product of a unit and an $\ell^{\rm th}$ power.
Since $D_s$ is a regular curve on $\XX$, $A/(\delta_s)$ is a semi local regular ring and hence
$A/(\delta_s)$ is an UFD. In particular $a_s = w_sb_s^\ell$ for some $w_s, b_s \in A/(\delta_s)$ 
with $w_s$ a unit. 
Since $\pi' = 1 \in A/(\delta)$, $a_s = \pi' \in A/(\delta_s)$ and $a_s = w_sb_s^\ell$,  we have 
$w_s$ is an $\ell^{\rm th}$ power in $A/(\delta_s,   \delta)$.  Hence 
$w_s = w_s'^{\ell} \in A/(\delta_s, \delta)$ for some unit $w_s' \in A/(\delta_s, \delta)$.
Since $D_j$'s  have normal crossings, the image of the ideal $(\delta_s, 
\delta)$ in $A/(\delta)$ is either the unit ideal or a maximal ideal.
Thus, by the Chinese remainder theorem,  there exists a unit $w_\delta \in A/(\delta)$ 
such that $w_s = w_\delta^\ell \in A/(\delta_s, \delta)$ for all 
$s$.  By (\ref{saltman_exact_seq}), there exists $w'  \in A$   such that 
$w' = w_s \in A/(\delta_s)$,  $w'  = w _{\delta}^\ell  \in A/(\delta)$.

Then, by the Chinese remainder theorem, there exists $w \in A$ such that 
 $w = 1 \in \kappa(P)$ for all $P \in \PP \setminus (\cup_1^n D_i)$  and $w = w' \in A/(\delta \prod_{m+1}^n \delta_s)$.
 Since $w_s \in A/(\delta_s)$, $w_\delta \in A/(\delta)$ are units, 
 $w$ is a unit in $A$. 

 Let    $u = w^{-1}u_1^{-1}$. 
 Since $u_1$ and $w$ are units in $A$, $u \in A$ is a unit. 
 We have  $u\prod \pi_i = w^{-1}\pi' \equiv  w_s^{-1}a_s = b_s^{\ell}$ modulo $ \delta_s$ for  $m+1 \leq s \leq n$
 and  $u\prod \pi_i = w^{-1}\pi' = w_\delta^{-\ell} \in A/(\delta)$. Since $\nu_{D_j}(\ell) = 0$ for 
 $m+1 \leq j \leq n$ (assumption A2)),      $u\prod\pi_i$ is an $\ell^{\rm th}$ power in $A/(\delta_j^{n_j})$ for 
 $1 \leq j \leq n$.
 Since $n_j = \ell \nu_{D_j}(\delta_j) + 1$, $u\prod \pi_i \in F_{D_j}^\ell$ for all $j$  (cf. \cite[\S 0.3]{epp}). 
 \end{proof}

 Let $u \in A$ be a unit  as in (\ref{square}) and $\pi = u\prod_1^r \pi_i \in A$.
Then $div_{\XX}(\pi) = \sum C_i + \sum_1^d t_sE_s$ 
for some irreducible curves $E_s$ with $E_s \cap \PP = \emptyset$.
 In particular $C_i \neq E_s$,  $D_j \neq E_s$ for all $i$, $j$ and $s$. 
Let $\PP'$ be a finite set of points of $\XX$ containing $\PP$,  $C_i \cap E_s$, $D_j \cap E_s$ for all $i$, $j$ and $s$ 
and at least one point from each $E_s$.   
Let $A'$ be the semi local ring at $ \PP'$.  For $1 \leq i \leq n$, 
let $\delta_i' \in A'$ be a prime defining $D_i$ on $A'$. Note that $\delta_i'A' \cap A = \delta_iA$ for all $i$.

 \begin{lemma}
\label{lift_of_residues}  
There exists $v \in A'$  such that \\
$\bullet$ $v$ is a unit and $F(\sqrt[\ell]{v})/F$ is unramified at all the points $P  \in  \PP' \cap (\cup_1^n D_i)$   
except possible at the points $P$  in $ D_i \cap D_j$ for all $i \neq j$ 
 with char$ (\kappa(P)) \neq \ell$ \\
$\bullet$ if char$(\kappa(D_j)) \neq \ell$,  
then the extension  $F(\sqrt[\ell]{v})/F$ is unramified at $D_j$ with the 
 reside field of $F(\sqrt[\ell]{v})$ at $D_j$ is equal to $\partial_{D_j}(\alpha)$ \\
 $\bullet$ if char$(\kappa(D_j)) = \ell$,  
 then  $F_{D_i}(\sqrt[\ell]{v}) \simeq F_{D_i}(\sqrt[\ell]{a})$. In particular 
  $\alpha \otimes F_{D_j}(\sqrt[\ell]{v})$ is trivial
\end{lemma}

\begin{proof}   
Let $1 \leq i \leq n$.
Suppose  char$(\kappa(P)) = \ell$ for all $P \in D_i$. 
If char$(\kappa(D_i)) \neq \ell$,  let $w_i \in \kappa(D_i)$ be such that $\partial_{D_i}(\alpha) = [w_i)$. 
Suppose cjar$(\kappa(D_i)) = \ell$.
By the assumption A10),   $\frac{a-1}{(\rho-1)^\ell}\in A_P$ for all $P \in D_i$.
In particular $\frac{a-1}{(\rho-1)^\ell}$ is regular at $D_i$ and 
the image of  $\frac{a-1}{(\rho-1)^\ell}$  in $\kappa(D_i)$ is in $A'/(\delta'_i)$.  
Let $w_i$ be the image of $\frac{a-1}{(\rho-1)^\ell} $ in $A'/(\delta'_i)$. 

Suppose there exists $P \in D_i$ with char$(\kappa(P)) \neq \ell$.
Then char$(\kappa(D_i)) \neq \ell$. 
 By (\cite[Proposition 7.10]{S3}), there exists $w_i \in \kappa(D_i)^*$
 such that \\
 $\bullet$ $\partial_{D_j}(\alpha) = \kappa(D_i)(\sqrt[\ell]{w_i})$,\\
 $\bullet$   $w_i$ is defined at all $P \in \PP \cap D_i$ with char$(\kappa(P)) \neq \ell$, \\
 $\bullet$  $w_i$ is a unit at all $P \in (\PP \cap D_i) \setminus (\cup_{j \neq i} D_j)$ with char$(\kappa(P)) \neq \ell$ \\
 $\bullet$ $w_i(P) =  w_j(P)$ for all $P \in D_i \cap D_j$, $i \neq j$ with char$(\kappa(P)) \neq \ell$. \\
 Let  $P \in \PP \cap D_i$. Suppose   char$(\kappa(P))  = \ell$.  Then, by the assumptions A10),  
 $[a)$ is unramified at $P$ (cf. \ref{epp}). Since $\alpha = [a, b)$, $\partial_{D_i}(\alpha) = [a(D_i))^{\nu_{D_i}(b)}$.
 In particular $\partial_{D_i}(\alpha) = \kappa(D_i)(\sqrt[\ell]{w_i})$ is unramified at $P$.
 Thus, by (\ref{localift}), we assume that $\frac{1-w_i}{(\rho-1)^\ell}$ 
 is unit at all $P \in \PP\cap D_i \setminus (\cap_{j \neq i}D_j)$ with char$(\kappa(P)) = \ell$. 
 Since char$(\kappa(D_i)) \neq \ell$, by assumptions A3) and A4), if $P \in D_i \cap D_j$ for some
 $j \neq i $, then char$(\kappa(P)) \neq \ell$. 
Thus  $\frac{1-w_i}{(\rho-1)^\ell} \in A'/(\delta_i)$.

Let $P \in D_i \cap D_j$ for some $i \neq j$. 
Suppose char$(\kappa(P)) = \ell$. Then, by the assumption A3) and A4), 
char$(\kappa(D_i)) = $ char$(\kappa(D_j)) = \ell$ and 
by the choice of $w_i$, we have  $w_i(P)   =  w_j(P)  \in  \kappa(P)$.
Suppose char$(\kappa(P)) \neq \ell$. Then, by the choice of $w_i$, we have $w_i = w_j  \in \kappa(P)$.
 
Let $u_i = w_i \in A'/(\delta'_i)$ if char$(\kappa(D_i)) = \ell$ and $u_i =   (1 - w_i)/(\rho-1)^\ell \in A'/(\delta'_i)$
if char$(\kappa(D_i)) \neq \ell$.  By   the assumptions A3) and A4), $u_i$ is defined at all 
$P \in \PP \cap D_i$ for all $i$
and  $u_i(P) = u_j(P) \in \kappa(P)$ for all $P \in D_i \cap D_j$ for $i \neq j$.
 Thus, by (\ref{saltman_exact_seq}), there exists $u' \in A'$ such that $u' = u_i $ modulo $(\delta'_i)$ for all $i$.
By the Chinese remainder theorem, we get $v' \in A'$ such that $v' = u' \in A'/(\prod \delta'_i)$ and
$v' = 0 \in \kappa(P)$ for all  $P \in \PP'$  with $P \not\in D_i $ for all $i$. 

We now show that $v = 1 - (\rho-1)^\ell v'$ has all the required properties. 

Let $P \in \PP'$. 
Suppose char$(\kappa(P)) = \ell$. Then $\rho -1 \in M_P$. Since $v' \in A'$, 
$v$ is a unit at $P$ and  $F(\sqrt[\ell]{v})$ is unramified at $P$ (\ref{epp_general}).
Suppose char$(\kappa(P)) \neq \ell$. 
Suppose that $P \not\in D_i$  for all $i$.
Then, by the choice of $v'$, $v' \in M_P$ and hence $v$ is a unit at $P$ and  $F(\sqrt[\ell]{v})/F$ is 
unramified at $P$. 
Suppose that $P \in D_i$ for some $i$.  Since char$(\kappa(P)) \neq \ell$,  char$(\kappa(D_i)) \neq \ell$.
Thus, by the choice of $v'$, we have $v' = u' = u_i = (1 - w_i)/(\rho-1)^\ell \in A'/(\delta'_i)$.
Hence $v = w_i \in A'/(\delta'_i)$. Suppose $P \not\in D_j $ for all $j \neq i$.
Then, by the choice $w_i$ is a unit at $P$ and hence $v$ is a unit at $P$.  
In particular $F(\sqrt[\ell]{v})/F$ is unramified at $P$.   
Thus  $v$ is a unit  and $F(\sqrt[\ell]{v})/F$ is unramified at all $ P\in \PP'$ except possibly at
 $P \in  D_i \cap D_j$ with char$(\kappa(P)) \neq \ell$. 

Suppose char$(\kappa(D_i)) \neq \ell$. 
Then, by the choice of $v$, we have 
$v = 1 - (\rho-1)^\ell v' = 1 - (\rho-1)^\ell u_i = w_i \in A'/(\delta'_i) \subset \kappa(D_i)$.  Since 
$w_i \neq 0$,  $v$ is a unit at $\delta_i$ and  
$F(\sqrt[\ell]{v})$ is unramified at $D_i$ with residue  field $\kappa(D_i)(\sqrt[\ell]{w_i}) = \partial_{D_i}(\alpha)$.

Suppose that char$(\kappa(D_i)) = \ell$. Since $v = 1 - (\rho-1)^\ell v'$ and $v' = u_i = w_i \in A'/(\delta'_i)$,  
$F(\sqrt[\ell]{v})$ is unramified at $D_i$ with residue field equal to  
$\kappa(D_i)[X]/(X^\ell - X + w_i) $ (\ref{epp}).
Since $ w_i$ is the image of $\frac{a-1}{(\rho-1)^\ell} $  in $A'/(\delta'_i)$,  the residue field of $F(\sqrt[\ell]{a})$ at $\delta'_i$ is 
$\kappa(D_i)[X]/(X^\ell - X + w_i) $ (\ref{epp}).  Hence  $F_{D_i}(\sqrt[\ell]{v}) \simeq   F_{D_i}(\sqrt[\ell]{a})$.
Since $\alpha = [ a, b )$, $\alpha \otimes F_{\delta'_i}(\sqrt{v}) $ is trivial. 
\end{proof}

\begin{remark} If $\ell$ is a unit in $A'$, then the extension $F(\sqrt[\ell]{v})/F$ given in the  above  lemma is 
the lift of the residues of $\alpha$ which is proved in (\cite[Proposition 7.10]{S3}).
\end{remark}

Choose $u \in A$ as in (\ref{square}) and $\pi = u\prod \pi_i \in A$
 and  $div_{\XX}(\pi) = \sum C_i + \sum_1^d t_sE_s$ 
for some irreducible curves $E_s$ with $E_s \cap \PP = \emptyset$
(see the paragraph before \ref{lift_of_residues}).  
Let $v \in A'$  as in (\ref{lift_of_residues}).
Let $V_1, \cdots, V_q$ be the irreducible curves in $\XX$ where $F(\sqrt[\ell]{v\pi})$ is ramified.
Since $\pi \in F_{D_j}^\ell$  (\ref{square}) and $F(\sqrt{v})$ is unramified at $D_j$ (\ref{lift_of_residues}) for all $j$,
$V_i \neq D_j$ for all $i$ and $j$.
Let $ \PP'' = \PP \cup (\cup (D_i \cap E_s))  \cup  (\cup (D_i \cap V_j)  $. 
 After reindexing $E_s$, 
we assume that there exists $d_1 \leq d$ such that  $E_s \cap\PP'' \neq \emptyset$ for $1 \leq s \leq d_1$
 and $E_s \cap\PP''  = \emptyset$ for $d_1 + 1 \leq s \leq d$. 
 
 \begin{lemma} 
 \label{norm}
Let $A''$ be the regular semi local ring at $ \PP'' $.
Then there exists $h \in F^*$ which is a norm from the extension $F(\sqrt[\ell]{v\pi})$
such that div$_\XX(h) = - \sum_1^{d_1}t_iE_i   + \sum r_iE_i' $, where $E_j' \cap \PP'' = \emptyset$ for all $j$.
\end{lemma}

\begin{proof}
Let $L = F(\sqrt[\ell]{ v\pi})$  and $T$ be  the integral closure of $A''$ in $L$.  

Let $1 \leq s \leq d_1$ and $P \in  \PP'' \cap E_s$. Since  $E_s \cap \PP  = \emptyset$, 
$P \in D_i \cap E_s$ for some $i$. 
Since $v$ is a unit at all $P \in (\PP' \setminus \PP)$ (\ref{lift_of_residues}) and $D_i \cap E_s \subset \PP'$, 
$v$ is a unit at $P$ and hence $v$ is a unit at $E_s$.   

Suppose that $t_s$ is coprime to $\ell$. 
Since $div_{\XX}(\pi) = \sum C_i + \sum_1^d t_sE_s$  and $v$ is a unit at $E_s$, 
 there is a unique  curve $\tilde{E}_s$  in $T$ lying over $E_s$. Let $t'_s = t_s$. 
  
Suppose that $t_s = \ell t'_s$. Let $\tilde{E}_s =  t'_s \sum E_{s,i}$, where $E_{s,i}$ are 
the irreducible divisors in $T$ which lie over $E_s$.  
Let $\tilde{E} = -\sum t'_s \tilde{E}_s$. 

 We claim that $ \tilde{E}$  is a principal divisor on $T$. 
Since $T$ is normal it is enough to check this at every maximal ideal of $T$. 
Let $M$ be a maximal ideal of $T$. Then  $ M \cap A'' = M_P$ for 
some $P \in   \PP''$. Suppose $P \not\in E_s$ for all $1 \leq s \leq d_1$. 
Then   $\tilde{E}$ is  trivial at $M$.  
Suppose that $P \in E_s$ for some $s$ with $1 \leq s \leq d_1$.  
Then, as we have seen above,  $P \in  D_i \cap E_s$ for some $i$.
Since $D_i \cap C_j \in \PP$ for all $i$ and $j$ and $\PP \cap E_s = \emptyset$, 
$P   \not\in C_i$ for all $i$. 
Hence div$_{A_P}(\pi) = \sum _{P \in E_i} t_iE_i$.
Since $v$ is a unit at $P$ (\ref{lift_of_residues}), div$_{A_P}(v\pi) = $ div$_{A_P}(\pi)$
and hence $\tilde{E}  =$  $-$div$(\sqrt[\ell]{v\pi})$ at $M$.
In particular $\tilde{E}$ is principal at $M$. 
Hence $\tilde{E} = div_T(g)$ for some $g \in L$. Let $h = N_{L/F}(g)$. Then div$_{A''}(h) = 
-\sum_1^{d_1} t_iE_i$ and hence $h$ has the required properties. 
\end{proof}

 \begin{lemma}
\label{norm1} Let $h \in F^*$ be as in (\ref{norm}) with 
div$_\XX(h) = - \sum_1^{d_1}t_iE_i + \sum r_jE_j'$. Then $\alpha$ is unramified at $E_j'$.
Further, if $r_j$ is coprime to $\ell$ for some $j$,  then the  specialization of 
$\alpha$ at $E_j'$ is unramified at every discrete valuation of $\kappa(E_j')$ which is centered on $E_j'$. 
\end{lemma}

\begin{proof} Since $E_j' \cap \PP'' = \emptyset$ and $D_i \cap \PP''  \neq \emptyset$ for all $i$, $E_j' \neq D_i$ for all $i$.
 Hence,  by the assumption A2), $\alpha$ is unramified at $E_j'$. 

Let $P$ be a closed point of $ E_j'$ for some $j'$ with $r_j$ coprime to $\ell$.
Let $L = F(\sqrt[\ell]{v\pi})/F$ and $B_P$ be  the integral closure   of $A_P$ in $L$.
We first  show that $\alpha \otimes_F L$ is unramified on $B_P$.

Suppose $P \not\in D_i$ for all $i$.
Then $\alpha$ is unramified at $P$ (assumption A2)).  
Hence there exists an Azumaya algebra $\AA_P$ over $A_P$ such that  the class of 
$\AA_P \otimes_{A_P} F$ is  $\alpha$ (cf. \ref{2dim_unramified}). 
In particular $\alpha \otimes_F L$ is the class of $\AA_P \otimes_{A_P} B_P$ and hence 
$\alpha \otimes_F L$ is unramified on $B_P$.

Suppose $P \in D_i$ for some $i$.  
Since $E_j' \cap \PP'' =  \emptyset$ (\ref{norm}), $P \not\in \PP''$.
 Since $\cup (V_{i'} \cap D_i) \subset \PP''$,   $P \not\in \cup V_{i'}$ for all $i'$
and hence $L$ is unramified at $P$.   
Hence $B_P$   is a regular semi local domain. 
Let $Q \subset B_P$ be a hight one prime ideal and $Q_0 = Q \cap A_P$. 
Then $Q$ is a height one prime ideal of $A_P$. 
If  $\alpha$ is unramified at  $Q_0$, then $\alpha \otimes_F L $ is unramified at $Q$.
Suppose that $\alpha$ is ramified at $Q_0$. Since $P \not \in D_j$ for $j \neq i$, 
 $Q_0 $   is the prime ideal corresponding to $D_i$. 
 Since  $\pi \in F_{D_i}^\ell$ (\ref{square}),
$F_{D_i}(\sqrt[\ell]{v\pi}) = F_{D_i}(\sqrt[\ell]{v})$. 
Suppose that char$(\kappa(D_i)) \neq \ell$.
Since $L/F$ is unramified at $D_i$ with residue field equal to $\partial_{D_i}(\alpha)$ (\ref{lift_of_residues}), 
$\alpha \otimes L/F$  is unramified at $Q$  (cf. \cite[Lemma 4.1]{PPS}).
Suppose that char$(\kappa(D_i)) = \ell$.
Since $\alpha \otimes F_{D_i}(\sqrt[\ell]{v})$ is trivial (\ref{lift_of_residues}),  $\alpha \otimes_F L$ is unramified at $Q$.  
Since $B_P$ is a regular semi local ring of dissension two,  $\alpha \otimes F(\sqrt[\ell]{v\pi})$ is
 unramified at $B_P$ (cf. \cite[Lemma 3.1]{LPS}).
 Let $\AA_P$ be an Azumaya algebra over $B_P$ such that $\alpha \otimes_F L$ is the class of
$\AA_P \otimes_{B_P} L$. 

Since $E_j$ is in the support of $h$, $r_j$ is coprime to $\ell$ and $h$ is a norm 
from $F(\sqrt[\ell]{v\pi})$,  $v \pi \in F_{E_j}^{\ell}$.  Let $\theta \in A_P$ be a prime defining $E_j$ at $P$. 
Then $\theta = \theta_1 \cdots \theta_\ell$ for distinct primes $\theta_i \in B_P$
and $A_P/(\theta) \subset B_P/(\theta_s) \subset \kappa(E_j)$ for all $s$.
Let $\beta \in H^2(\kappa(E_j'),\Z/\ell(1))$ be the specialization of $\alpha$ at $E_j'$.
Then $\beta$ is the class of $\AA_P \otimes_{B_P/(\theta_i)} \kappa(E_j)$.
Hence $\beta$ is unramified at $B_P/(\theta_s)$.
Since $B_P/(\theta_s)$ is integral over $A/(\theta)$, 
  $\beta$ is unramified at all discrete valuations of $\kappa(E_j)$ which 
are centered at $A_P/(\theta)$.
\end{proof}

\begin{theorem}
\label{local_global} Suppose $(\XX, \zeta, \alpha)$ satisfies the assumption in \ref{assumptions}. 
 Then there exists $f \in K^*$ such that for every   $x \in \XX_{(1)}$, 
$ \partial_x(\zeta - \alpha \cdot (f))$ is unramified at every discrete valuation of $\kappa(x)$ centered 
on the closure of $\{ x \}$. 
\end{theorem}
 
\begin{proof} Let $\PP$ be a finite set of closed points of $\XX$ containing $C_i \cap C_j$ for all $i \neq j$,
$D_i \cap D_j$ for all $i \neq j$, $C_i \cap D_j$ for all $i$ and $j$  and 
at least one point from each $C_i$ and $D_j$. 
Let $A$ be the semi local ring at $\PP$. Let $\pi_i  \in A$ be  primes defining $C_i$.

Let $u \in A$ be a unit in $A$ as in (\ref{square}) and $\pi = u\prod_1^r \pi_i \in A$.
Then $div_{\XX}(\pi) = \sum C_i + \sum_i^d t_iE_i$ with   $E_s \cap  \PP = \emptyset$. 
 In particular $C_i \neq E_s$,  $D_j \neq E_s$ for all $i$, $j$ and $s$. 
Let $\PP'$ be a finite set of points containing $\PP$,  $C_i \cap E_s$,  $D_j \cap E_s$ for all $i$, $j$ and $s$ 
and at least one point from each $E_s$.   
Let $A'$ be the semi local ring at $ \PP'$.  Let $v \in A'$ be as in  (\ref{lift_of_residues}).

Let $V_1, \cdots, V_q$ be the irreducible curves in $\XX$ where $F(\sqrt[\ell]{v\pi})$ is ramified.
Let $ \PP'' = \PP \cup (\cup (D_j \cap E_s)) \cup  (\cup (D_j \cap E_s))$. After reindexing $E_s$, 
we assume that there exists $d_1 \leq d$ such that  $E_s \cap\PP'' \neq \emptyset$ for $1 \leq s \leq d_1$
 and $E_s \cap\PP''  = \emptyset$ for $d_1 + 1 \leq s \leq d$. 
Let $A''$ be the regular semi local ring at $ \PP''$.  Let $h \in F^*$ be as in (\ref{norm}). 
  
We claim that $f = h \pi$ has the required properties, i.e. $ \partial_{x}(\zeta - \alpha \cdot (f))$ is unramified 
at every discrete valuation of $\kappa(x)$ for all   $x \in \XX_{(1)}$. 

Let $x \in \XX_{(1)}$ and  $D$  be the closure of $\{ x \}$. Suppose $D = C_i$ for some $i$.
Then $h$ is a unit at $C_i$ (\ref{norm}),   $\alpha$ is unramified at $C_i$ (assumption A2)) 
and $\pi$ is a parameter at $C_i$, we have  $\partial_{C_i}(\alpha \cdot (f))$ is the specialization of 
$\alpha$ at $C_i$ (\ref{residue}). Hence, 
 by the assumption A6),  $\partial_{C_i}(\zeta - \alpha \cdot (f)) = 0$. 
 
Suppose that $D  = D_j$ for some $j$.   By the assumption A2),  $\partial_{D_j}(\zeta)  = 0$.
Suppose $\alpha$ is unramified at $D_j$. Since $\pi$ and $h$ are  units at $D_j$, 
$\partial_{D_j}(\alpha \cdot (f)) = 0$ (\ref{residue}). 
Suppose $\alpha$ is ramified at $D_j$. If char$(\kappa(D_j)) = \ell$, then 
by the choice $\alpha \otimes F_{D_j}(\sqrt[\ell]{v}) = 0$ (\ref{lift_of_residues}).
Suppose that char$(\kappa(D_j)) \neq \ell$. 
Since $F_{D_j}(\sqrt[\ell]{v})$ is unramified with residue field equal to $\partial_{D_j}( \alpha)$
(\ref{lift_of_residues}),  we have $\alpha \otimes  F_{D_j}(\sqrt[\ell]{v}) = 0$ (\ref{ind_ell}). 
In particular, in either  case, $\alpha \cdot (g) = 0   \in H^3( F_{D_j}(\sqrt{v}), \Z/\ell(2))$. 
Since $\pi \in F_{D_i}^\ell$  (\ref{square}),  $L \otimes F_{D_j}  = F_{D_j}(\sqrt[\ell]{v})$
and $\alpha \cdot (\pi) = 0\in H^3(F_{D_j}, \Z/\ell(2))$.
Thus $\alpha \cdot (h) = $ cor$_{L/F}(\alpha \cdot (g)) = 0 \in H^3(F_{D_j}, \Z/\ell(2))$
and  $\partial_{D_j}(\alpha \cdot (h)) = 0$.  Hence $\partial_{D_j}(\zeta - \alpha \cdot (f)) =   0$.

Suppose $D \neq C_i$ and  $D_j$ for all $i$ and $j$. Then $\partial_D(\zeta) = 0$ and $\alpha$ is unramified at $D$.
If $\nu_D( f) $ is a multiple of $\ell$, then $\partial_D( \alpha \cdot (f)) = 0$. 
Suppose that $\nu_D(f)$ is coprime to $\ell$.  Since
$div_{\XX}(\pi) = \sum C_i + \sum_1^dt_iE_i$ (\ref{square}),
div$_\XX(h) = - \sum_1^{d_1}t_sE_s + \sum r_iE'_i$  (\ref{norm}) and $f = h\pi$,
we have 
div$_\XX(f) = \sum C_i +  \sum_{d_1+1}^d t_s E_s + \sum r_iE_i' $.
Since $\nu_D(f)$ is coprime to $\ell$ and $D \neq C_i$ for all $i$, $D = E_s$ for some $ d_1 +1 \leq s \leq d$
or $D = E'_i$ for some $i$. 

If $D = E'_i$, then by (\ref{norm1}), $\overline{\alpha}$ is unramified at every discrete 
valuation of $\kappa(D)$ centered on $D$. 
Suppose $D = E_s$ for some $d_1 +1 \leq s  \leq d$.  Then by the choice of $d_1$,
$E_s \cap \PP'' = \emptyset$ and hence $E_s \cap D_j = \emptyset$ for all $j$.   
Let $P \in E_s$. Then $\alpha$ is unramified at $P$ (assumption A2)) and hence
$\overline{\alpha}$ is unramified at $P$. In particular $\overline{\alpha}$ is   unramified at every 
discrete valuation of $\kappa(E_s)$ centered at $P$.
Since $\alpha$ is urnamified at $E_s$,  $\partial_{E_s}(\alpha \cdot (f)) = 
\overline{\alpha}^{\nu_{E_s}(f)}$ (\ref{residue}). Since $\overline{\alpha}$ is unramified at
 every discrete valuation of $\kappa(E_s)$ centered on $E_s$, $\partial_{E_s}(\alpha \cdot (f))$
 is unramified at every discrete valuation of $\kappa(E_s)$ centered on $E_s$. 
 Hence $f$ has the required property.
 \end{proof}
  
  \section{ Divisibility of elements in $H^3$ by symbols in $H^2$} 
   Let $K$ be a global field  or a local field and $F$ the function field of  a curve over $K$. 
 If $K$ is a number field   or a local field, let $R$ be the ring of integers in $K$. 
 If $K$ is a global field of positive characteristic, let $R$ be the field of constants of $K$.
 Let $\XX$ be a regular proper model of $F$ over Spec$(R)$.  Let $\ell$ be a prime not equal to char$(K)$. 
 Suppose that $K$ contains a primitive $\ell^{\rm th}$ root of unity $\rho$.
 Then for any $P \in \XX_{(2)}$, $\kappa(P)$ is a finite field. Hence if char$(\kappa(P)) = \ell$, then  
$\kappa(P) = \kappa(P)^\ell$.

 Thus we have a complex  (cf. \ref{complex})
$$
0 \to H^3(F,  \Z/\ell(2))  \buildrel{\partial}\over{\to} \oplus_{x \in \XX_{(1)}}H^2(\kappa(x),  \Z/\ell(1)) 
\buildrel{\partial}\over{\to} \oplus_{P\in \XX_{(2)}} H^1(\kappa(P), \Z/\ell).
$$

Let $\zeta \in H^3(F, \Z/\ell(2))$ and $\alpha = [a, b) \in H^2(F, \Z/\ell(1))$.
In this section we prove (see \ref{localglobal}) a certain  local global principle for divisibility of $\zeta$ by $\alpha$
if $(\XX, \zeta, \alpha)$ satisfies certain  assumptions  (see \ref{assumptions2}).

For a sequence of blow-ups $\eta : \YY \to \XX$ and for an irreducible curve $C$ in $\XX$, 
we denote the strict transform of $C$ in $\YY$ by $C$ itself.

We begin with the following

\begin{lemma} 
\label{blowups1}
Suppose $(\XX, \zeta, \alpha)$ satisfies the assumption A1) of   \ref{assumptions}.
Let $\YY \to \XX$ be a sequence of blow ups centered on closed points of $\XX$ which are not in 
$C_i \cap C_j$ for all $i \neq j$.
Let $1 \leq I \leq 11$ with  $I \neq 3, 5, 7$. 
 If $(\XX,  \zeta, \alpha)$ satisfies the    assumption AI) of
\ref{assumptions}, then $(\YY,  \zeta, \alpha)$ also satisfies the 
assumption AI).
\end{lemma}

\begin{proof} Let $Q $ be a closed point  of $\XX$ which is not in $C_i \cap C_j$ for $i \neq j$
and   $\eta: \YY \to \XX$   a simple blow-up at $Q$.   It is enough to prove the lemma for
 $(\YY, \zeta, \alpha)$. 
 
Let $E$ be the exceptional curve in $\YY$.  Since $Q \not\in C_i \cap C_j $ for $i \neq j$ and 
$(\XX, \zeta, \alpha)$ satisfies A1) of \ref{assumptions}, 
by (\ref{curve-point}), $\zeta$ is unramfied at $E$.

Let $1 \leq I \leq 11$ with  $I \neq 3, 5, 7$. 
Suppose further $I \neq 4, 10$.
Since the exceptional curve $E$ is not in ram$_{\YY}(\zeta)$,
if $(\XX, \zeta, \alpha)$ satisfies the assumption AI) of \ref{assumptions}, 
then  $(\YY, \zeta, \alpha)$ also satisfies the same  assumption.

Suppose $(\XX,  \zeta, \alpha)$ satisfies the assumption A4)  of \ref{assumptions}.
Suppose char$(\kappa(Q)) = \ell$. Then char$(\kappa(E)) = \ell$ and hence
$(\YY, \zeta, \alpha)$ also satisfies the 
assumption A4)  of \ref{assumptions}.  Suppose char$(\kappa(Q)) \neq \ell$. 
Then char$(\kappa(P)) \neq \ell$ for all $P \in E$ and hence
$(\YY, \zeta, \alpha)$ also satisfies the  assumption A4)  of \ref{assumptions}.

Suppose $(\XX, \zeta, \alpha)$ satisfies the assumption A10) of \ref{assumptions}.
If char$(\kappa(Q)) \neq \ell$, then char$(\kappa(P)) \neq \ell$  for all $P \in E$ and hence 
$(\YY, \zeta, \alpha)$ also satisfies the assumption A10) of \ref{assumptions}.
Suppose that char$(\kappa(Q)) = \ell$. If $Q \not\in  D_i$ for any  $i$, 
then $\alpha$ is unramified at $Q$ and hence $\alpha$ is unramified at $E$.
In particular $E \not\in $ ram$_\YY(\alpha)$ and hence $(\YY, \zeta, \alpha)$ also satisfies the assumption A10) of \ref{assumptions}.
 Suppose $Q \in D_i$ for some $i$. Since $(\XX, \zeta, \alpha)$ satisfies A10) of \ref{assumptions}, 
$\frac{1-a}{(\rho-1)^\ell} \in A_Q$.
Let $P \in E$. Since $A_Q \subset A_P$,   $\frac{1-a}{(\rho-1)^\ell} \in A_P$.
Hence $(\YY, \zeta, \alpha)$ also satisfies the assumption A10) of \ref{assumptions}.
\end{proof}

\begin{lemma} 
\label{blowups2}
Let $\YY \to \XX$ be a sequence of blow ups centered on closed points $Q$ of $\XX$ with char$(\kappa(Q)) \neq \ell$. 
Suppose $(\XX, \zeta, \alpha)$ satisfy the assumptions A1) and A2).
If $(\XX,  \zeta, \alpha)$ satisfies the    assumption A3)  or  A7) of
\ref{assumptions}, then $(\YY,  \zeta, \alpha)$ also satisfies the  same assumption. 
\end{lemma}

\begin{proof} Let $Q $ be a closed point  of $\XX$  with char$(\kappa(Q))  \neq \ell$ 
and $E$ the exceptional curve in $\YY$.  Since char$(\kappa(E)) \neq \ell$ and 
for any closed point $P$ of $E$ char$(\kappa(P)) \neq \ell$, 
the lemma follows. 
\end{proof}

\begin{assumptions}
\label{assumptions2} Suppose $(\XX, \zeta, \alpha)$ satisfies the following.
\begin{enumerate} 
\item[{B1)}] ram$_\XX(\zeta)  = \{ C_1, \cdots ,   C_r \}$,   
$C_i$'s are  irreducible  regular curves with normal crossings \\ 
\item[{B2)}] ram$_\XX(\alpha) =  \{D_1, \cdots , D_n \}$ with $D_j$'s irreducible curves 
such that  $C_i \neq D_j$ for all $i$ and $j$  \\
\item[{B3)}]  if $D_s \cap C_i \cap C_j \neq  \emptyset$ for some $s$, $i \neq j$, then char$(\kappa(D_s)) \neq \ell$ \\
\item[{B4)}] if  $P \in D_j$ for some $1 \leq j \leq n$  with char$(\kappa(P)) = \ell$,  then $\frac{1-a}{(\rho-1)^\ell} \in A_P$ \\
\item[{B5)}]   $\partial_{C_i}(\zeta)$ is the specialization of $\alpha$ at $C_i$ for all $i$ \\
\item[{B6)}] if $\ell  = 2$, then $\zeta \otimes F \otimes K_\nu$ is trivial for all real places $\nu$ of $K$ \\
\item[{B7)}] if $\ell = 2$, then $a$ is a sum of two  squares in $F$ \\
\item[{B8)}] for  $1 \leq i < j \leq r$, there exists at most one  $D_s$ with $D_s \cap C_i \cap C_j \neq \emptyset$
and if $P \in D_s \cap C_i \cap C_j$, then $D_s$ is defined  by  $u\pi_i^{\ell-1} + v\pi_j$ at $P$ for some units $u$ and $v$ at $P$ and 
$\pi_i, \pi_j$ primes defining $C_i$ and $C_j$ at $P$. 
\end{enumerate} 
\end{assumptions}

Let $\PP$ be a finite set of closed points of $\XX$ containing $C_i \cap C_j$, $D_i \cap D_j$ for all $i \neq j$, 
$C_i \cap D_j$ for all $i, j$ 
and at least one point from each $C_i$ and $D_j$. Let $A$ be the regular semi local ring at $\PP$ on $\XX$.
For every $P \in \PP$, let   $M_P$ be the maximal ideal of $A$ at $P$.  
For  $1 \leq i \leq r$ and $1 \leq j \leq n$, let $\pi_i \in A$ be a prime defining $C_i$ on $A$ and $\delta_j \in A$ 
a prime defining $D_j$ on $A$.

\begin{lemma} 
\label{blowups}
Suppose  $(\XX, \zeta, \alpha)$ satisfies  the assumptions  \ref{assumptions2}.
Let $\YY \to \XX$ be a sequence of blow ups centered on closed points of $\XX$ which are not in 
$C_i \cap C_j$ for $i \neq j$. Then $(\YY, \zeta, \alpha)$ also  satisfies the assumptions \ref{assumptions2}.
\end{lemma}

\begin{proof}  Let $Q $ be a closed point  of $\XX$ which is not in $C_i \cap C_j$ for $i \neq j$ 
and   $\eta: \YY \to \XX$   a simple blow-up at $Q$.   It is enough to show that $(\YY, \zeta, \alpha)$ satisfies the assumptions 
\ref{assumptions2}. 

Let $E$ be the exceptional curve in $\YY$.  Since $Q \not\in C_i \cap C_j $ for $i \neq j$, 
by (\ref{curve-point}), $\zeta$ is unramfied at $E$.
Hence  ram$_\YY(\zeta) = \{C_1, \cdots , C_r \}$ and  $(\YY, \zeta, \alpha)$ satisfies B1). 
We have ram$_\YY(\alpha) \subset \{D_1, \cdots , D_n, E\}$ and hence $(\YY, \zeta, \alpha)$ satisfies B2).
Since $E \cap C_i \cap C_j = \emptyset$ for all $i \neq j$, $(\YY, \zeta, \alpha)$ satisfies B3) and B8). 

Suppose $Q \in D_i$ for some $i$ with char$(\kappa(Q)) = \ell$. Then, by B4), $\frac{1-a}{(\rho-1)^\ell} \in A_Q$  and hence 
$\frac{1-a}{(\rho-1)^\ell} \in A_P$ for all   closed points $P$ of $E$. 
Suppose that  $Q \not\in D_i$ of any $i$. Then $\alpha$ is unramified at $Q$. 
In particular  $\alpha$ is unramified at $E$ and hence  $E \not\in $ ram$_\YY(\alpha)$. Thus
 and $(\YY, \zeta, \alpha)$   satisfies B4).

Since $E$ is not in ram$_\YY(\alpha)$,  $(\YY, \zeta, \alpha)$ satisfies B5). 
 
 Since B6) and B7) do  not depend on the model,    $(\YY, \zeta, \alpha)$ satisfies all the assumptions \ref{assumptions2}. 
\end{proof}

\begin{theorem} 
\label{localglobal}
Let $K$, $F$ and  $\XX$ be as above. Let $\zeta \in H^3(F, \Z/\ell(2))$ and $\alpha = [a, b) \in H^2(F, \Z/\ell(1))$.
Suppose that $F$ contains a primitive $\ell^{\rm th}$ root of unity. If 
 $(\XX, \zeta, \alpha)$ satisfies the assumptions \ref{assumptions2},
then there exists $f \in F^*$ such that $\zeta =  \alpha \cdot (f)$.
\end{theorem}

\begin{proof} Suppose  $(\XX, \zeta, \alpha)$ satisfies the assumptions \ref{assumptions2}.
First we show that there exists a sequence of blow ups $\eta : \YY \to \XX$ such that 
$(\YY, \zeta, \alpha)$ satisfies the assumptions \ref{assumptions}.

Let $P \in \XX_{(2)}$. Suppose $P \in D_s$  for some $s$ and $D_s$ is not regular at $P$ or  $P \in D_s \cap D_t$  
for some $s \neq t$.
Then, by the assumption B8),   $P \not\in C_i \cap C_j$ for all $i \neq j$. 
Thus,   there exists a   sequence of blow ups $\XX' \to \XX$ at closed points which are  not 
in $C_i \cap C_j$  for all $i \neq j$ such that  ram$_{\XX'}(\alpha)$ is  a union of  regular with 
normal crossings.  By (\ref{blowups}), $\XX'$ also satisfies the assumptions \ref{assumptions2}.
Thus,  replacing $\XX$ by $\XX'$ we assume that $(\XX, \zeta, \alpha)$ satisfies the assumptions 
\ref{assumptions2} and $D_i$'s are regular with normal crossings. 
In particular $(\XX, \zeta, \alpha)$ satisfies the assumptions A1) and A2) of \ref{assumptions}. 

Suppose there exists $i \neq j$ and $P \in D_i \cap D_j$ such that char$(\kappa(D_i)) \neq \ell$,
char$(D_j) \neq \ell$ and char$(\kappa(P))=\ell$. Let $\XX' \to \XX$ be the blow-up at $P$
and $E$ the exceptional curve in $\XX'$. Then char$(\kappa(E)) =$ char$(\kappa(P)) = \ell$ and
 $D_i  \cap D_j \cap E = \emptyset$ in $\XX'$. By the assumption B8), $P \not\in C_{i'} \cap C_{j'}$ for 
 all $i' \neq j'$ and hence $\XX'$ satisfies assumptions of \ref{assumptions2} (cf. \ref{blowups})
 and assumptions A1) and A2) of \ref{assumptions} (cf. \ref{blowups1}). Thus replacing $\XX$ by a sequence of  blow-ups 
 at closed points in $D_i \cap D_j$ for $i \neq j$, we assume that $\XX$ satisfies  assumptions of \ref{assumptions2}
 and  assumptions A1), A2) and A4) of \ref{assumptions}.

Since $(\XX, \zeta, \alpha)$ satisfies the   assumptions B4), B5) and B8) of \ref{assumptions2}
$(\XX, \zeta, \alpha)$  satisfies   the assumptions  A6), A9), A10) and A11) of \ref{assumptions}.

Suppose  $P \in C_i \cap D_s$ for some $i, s$ and $P \not\in C_j$ for all $j \neq i$.  
  Since $\zeta$ is unramified at $P$ except at $C_i$, $\partial_{C_i}(\zeta)$ is zero over  
 $\kappa(C_i)_P$ (\ref{curve-point}).
 By the assumption B5),  we have  $\partial_{C_i}(\zeta) = \overline{\alpha}$.
 Thus, by (\ref{curve_point_23}), $\alpha \otimes F_P  = 0$.  Let $\XX' \to \XX$ be the blow-up at $P$ and $E$ the exceptional curve in $\XX'$.
 Since $\alpha \otimes F_P = 0$ and $F_P \subset F_E$, $\alpha$ is unramified at $E$ and hence  
 ram$_{\XX'}(\alpha) = \{ D_1, \cdots , D_n \}$.
 Note that $C_i \cap D_s   = \emptyset$ in $\XX'$. 
 Hence $(\XX', \zeta, \alpha)$ satisfies  the assumption A8) of \ref{assumptions}. 
 Since $P \not\in C_j  $ for 
 all $j \neq i$,   $(\XX', \zeta, \alpha)$ satisfies the assumptions \ref{assumptions2} (\ref{blowups})
 and assumptions  \ref{assumptions} except  possibly A3),  A5)  and A7)  (\ref{blowups1}).
 Thus, replacing $\XX$ by $\XX'$ we assume that $(\XX, \zeta, \alpha)$ satisfies
 assumptions \ref{assumptions2} and   the assumptions    of \ref{assumptions} except possibly  A3),   A5) and A7).

  Let ram$_{\XX} (\alpha) = \{D_1, \cdots, D_m, D_{m+1}, \cdots , D_n \}$ with char$(\kappa(D_s)) = \ell$ for $ 1\leq s \leq m$
 and char$(\kappa(D_t)) \neq  \ell$ for $ m+1 \leq t \leq n$.  Suppose $D_s \cap D_t \neq\emptyset$ for some $1 \leq s \leq m$
 and $m+1 \leq t \leq n$.  Let $P \in D_s \cap D_t$. Then char$(\kappa(P)) = \ell$ and hence 
 $\frac{a-1}{(\rho-1)^\ell} \in A_P$ (assumption B4)).
In particular  $[a)$ is unramified at $P$ (cf.  \ref{epp}).  Since $\alpha$ is ramified at  $D_t$, 
 $\nu_{D_t}(b)$ is coprime to $\ell$ and hence there exists $i$ such that $\nu_{D_s}(b) + i\nu_{D_t}(b)$ is divisible by $\ell$. 
  Let $\XX_1 \to \XX$ be the blow-up at $P$ and $E_1$ the exceptional curve in $\XX_1$.
  We have   $\nu_{E_1}(b)  = \nu_{D_s}(b)  + \nu_{D_t}(b)$. 
  Let $Q_1$ be the point in $E_1 \cap D_t$ and $\XX_2 \to \XX_1$ be the blow-up at $Q_1$.
  Let $E_2$ be the exceptional curve in $\XX_2$. We have $\nu_{E_2}(b) = \nu_{E_1}(b) + \nu_{D_t}(b) = \nu_{D_s}(b) + 2 \nu_{D_t}(b)$.
  Continue this process $i$ times and get $\XX_i \to \XX_{i-1}$ and $E_i$ the exceptional curve in $\XX_i$.
  Then $\nu_{E_i}(b) = \nu_{D_t}(b) + i \nu_{D}(b)$ is divisible by $\ell$.  
  Since $[a)$ is unramified at $P$,     $\alpha$ is unramified at $E_i$. 
  Since char$(\kappa(E_j)) = \ell$ for all $j$, $E_{i-1} \cap D_t = \emptyset$ in $\XX_i$ and $E_i$ is in not
  in ram$_{\XX_i}(\alpha)$. Since $P \not\in C_i \cap C_j$ for all $i \neq j$ (assumption B4)), 
  $\XX_i$ satisfies assumptions \ref{assumptions2} (cf. \ref{blowups}).  Thus, replacing $\XX$ by $\XX_i$, we assume that 
  $D_s \cap D_t = \emptyset$ for all  $1 \leq s \leq m$
 and $m+1 \leq t \leq n$ and $\XX$ satisfies assumptions \ref{assumptions2}. 
 Thus $\XX$ satisfies all the assumptions of \ref{assumptions} expect possibly  A5) and A7) (cf. \ref{blowups1} ).    
 
Suppose $C_i \cap D_t \neq \emptyset$ for some $i$ and $t$. 
Since $(\XX, \zeta, \alpha)$ satisfies the assumptions A8) and A9) of \ref{assumptions}, 
there exists $j \neq i$ such that $C_i \cap C_j \cap D_t \neq \emptyset$. 
Since $(\XX, \zeta, \alpha)$ satisfies the assumption B3) of \ref{assumptions2}, 
char$(\kappa(D_t)) \neq \ell$. Hence $C_i \cap D_t = \emptyset$ for all $i$ and $1 \leq t \leq m$.
In particular $(\XX, \zeta, \alpha)$ satisfies the assumption A7) of \ref{assumptions}
and hence  $(\XX, \zeta, \alpha)$ satisfies  all the assumptions  of \ref{assumptions} except possibly   A5).

Let $P \in \XX_{(2)}$. Suppose that $P$ is a   chilly point for $\alpha$. 
Then $P \in D_s \cap D_t$ for some $D_s , D_t \in $ ram$_\XX(\alpha)$ with $D_s \neq D_t$
with char$(\kappa(P)) \neq \ell$.
In particular $P \not \in C_i \cap C_j$ for all $i \neq j$ (assumption B8)). 
Since there is a sequence of blow-ups $\YY \to \XX$ centered on chilly points of $\alpha$ on $\XX$
with no chilly loops on $\YY$ (\ref{chillyloops}), by (\ref{blowups1}, \ref{blowups2}), replacing $\XX$ by $\YY$,
we assume that $(\XX, \zeta, \alpha)$ satisfies assumptions \ref{assumptions2} and  \ref{assumptions}.

Thus, by (\ref{local_global}), 
there exists $f \in F^*$ such that for every $x \in \XX_{(1)}$, 
 $\partial_x(\zeta - \alpha \cdot (f))$ is unramified at every discrete valuation of 
$\kappa(x)$ centered at a closed point of the closure $\overline{\{ x \}}$ of $\{ x \}$. 
Since $\kappa(x)$ is a global field or a local field, every discrete valuation of $\kappa(x)$ is 
centered on a closed point of $\overline{\{ x \} }$.  Hence 
$\partial_x(\zeta - \alpha \cdot (f))$ is unramified at every discrete valuation of $\kappa(x)$. 
 
For   place $\nu$  of $K$, let  $K_\nu$ be the completion of $K$ at $\nu$ and
 $F_\nu = F\otimes_KK_\nu$. 
 
 Let $\nu$ be a real place of $K$.
 Since $a$ is a sum of  two squares in $F$, $a$ is a norm from the extension  $F_\nu(\sqrt{-1})$.
Let $\tilde{a} \in F_\nu(\sqrt{-1})$ with norm equal to $a$. 
Since $H^2(F_\nu (\sqrt{-1}), \Z/2(2)) = 0$ (\cite[p. 80]{Serre_Gal_coh}) 
and  cor$_{F_\nu(\sqrt{-1})/F_\nu}[\tilde{a}, b)   = [a, b) \otimes F_\nu$, 
$\alpha = [a, b) = 0 \in H^2(F_\nu, \Z/2(2))$.  
Since, by assumption $\zeta \otimes F_\nu = 0$, $\zeta - \alpha \cdot (f) = 0 \in H^3(F_\nu, \Z/2(2))$.

Let $x \in \XX_{(1)}$. Since $\zeta - \alpha \cdot (f) = 0 \in H^3(F_\nu, \Z/2(2))$ for all real places $\nu$ of $K$, 
it follows that $\partial_x(\zeta - \alpha \cdot (f)) = 0 \in H^2(\kappa(x)_{\nu'}, \Z/2(1))$ for all real places $\nu'$ of 
$\kappa(x)$.  Since $\partial_x(\zeta - \alpha \cdot (f))$ is unramified at every discrete valuation of $\kappa(x)$, 
$\partial_x(\zeta - \alpha \cdot (f)) = 0$ (cf. \cite[p. 130]{CFANT}). Hence $\zeta - \alpha \cdot (f)$ is unramified 
on $\XX$.

Let $\nu$ be a finite place of $K$.  Since $\zeta - \alpha \cdot (f)$ is unramified on $\XX$,  
 $(\zeta  - \alpha \cdot (f) ) \otimes _FF_\nu) = 0 \in H^3(F_\nu, \Z/\ell(2))$ (\cite[Corollary p. 145]{K2}).
Hence $\zeta = \alpha \cdot (f)$ (\cite[Theorem 0.8(2)]{K2}).
\end{proof}
 
 \section{Main theorem}
 
 In this section we prove our main result (\ref{main}).
 Let $K$ be a global field  or a local field and $F$ the function field of  a curve over $K$. 
  Let $\ell$ be a prime not equal to char$(K)$. Suppose that $F$ contains a primitive $\ell^{\rm th}$ root of unity $\rho$.
  If $K$ is a number field   or a local field, let $R$ be the ring of integers in $K$. 
 If $K$ is a global field of positive characteristic, let $R$ be the field of constants of $K$.

To prove our main result (\ref{main}), we first show   (\ref{prop}) that  given $\zeta \in H^3(F, \Z/\ell(2))$ with
 $ \zeta \otimes_F (F \otimes_K K_\nu) = 0$ for all real places $\nu$ of $K$, 
there exist $\alpha = [a, b) \in H^2(F, \Z/\ell(1))$ and  a regular proper model $\XX$ of $F$   over $R$ such that 
the triple $(\XX, \zeta, \alpha)$ satisfies the assumptions  \ref{assumptions2}.

Let $\zeta \in H^3(F, \Z/\ell(2))$ be such that $ \zeta \otimes_F (F \otimes_K K_\nu) = 0$ for all real places 
$\nu$ of $K$.
Choose a regular proper model $\XX$ of $F$ over $R$ (cf. \cite[p. 38]{S1}) such that  \\ 
$\bullet$  ram$_{\XX}(\zeta) \cup $ supp$_\XX (\ell) \subset \{ C_1, \cdots, C_{r_1}, \cdots , C_r \}$, where  $C_i$'s  are  irreducible 
regular curves with  normal crossings \\
$\bullet$ for $i \neq j$,  $C_i$ and $C_j$ intersect at most at one closed point \\
$\bullet$    $C_i \cap  C_j  = \emptyset$ if $i, j \leq r_1$ or $i, j > r_1$.

 For $x \in \XX_{(1)}$, let $\beta_x = \partial_x(\zeta)$.
Let $\PP_{0} \subset \cup C_i$ be a finite set of closed points of $\XX$ containing    $C_i \cap C_j$ for $1 \leq i <  j \leq r$, 
and  at least  one closed point from  each $C_i$. 
Let $A$ be the regular semi local ring at the points of $\PP_0$. 
Let $Q \in C_i$ be a closed point. Since $C_i$ is regular on $\XX$, $Q$ gives a discrete valuation 
$\nu^i_Q$ on $\kappa(C_i)$.

 \begin{lemma}
 \label{choice_of_a}  There exists   $a \in A$ such that  \\
 $\bullet$ $\frac{a-1}{(\rho-1)^\ell} \in A$ and $[a )$ is unramified on $A$ \\
$\bullet$ for   $ 1 \leq i \leq r_1$ and $P \in C_i \cap \PP_0$,  
 $ \partial_P(\beta_{x_i} )  = [a(P) )$    \\
 $\bullet$ for   $ r_1 + 1 \leq i \leq r$ and $P \in C_i \cap \PP_0$,  $ \partial_P(\beta_{x_i})   = [ a(P) ) ^{ -1}$ \\
 $\bullet$ if $P \in \PP_0$ and $P \not\in C_i \cap C_j$ for all $i \neq j$, then $[a(P))$ is the trivial extension \\
 $\bullet$  if $\ell = 2$, then $a$ is a sum of two squares in $A$.  
  \end{lemma}
 
 \begin{proof}   
 Let $P \in \PP_0$.  Suppose  $P \in C_i \cap C_j$ for some $i <  j$.
 Then, by the choice of $\XX$, the pair $(i, j)$ is uniquely determined by $P$.   
 Let $u_P \in \kappa(P)$ be such that  $\partial_P(\partial_{x_i}(\zeta))   = [ u_P )$.    
If $P \not \in C_i \cap C_j$ for all $i \neq j$, let $u_P  \in \kappa(P)$ with $[u_P)$ the trivial 
extension. 

Then, by (\ref{choice_of_a_general}), there exists $a \in A$ such that for 
every $P \in \PP_0$, the cyclic extension $[a)$ over $F$  is unramified on $A$ 
with  the residue field $[a(P))$ of $[a)$ at $P$ is $[u_P)$. Further if $\ell = 2$, choose  
$a$ to be a  sum of two squares in $A$ (\ref{choice_of_a_general}).
From the proof of (\ref{choice_of_a_general}), we have $\frac{a-1}{(\rho-1)^\ell} \in A$. 
 
Let $P \in \PP_0$.  
Suppose that $P \in C_i$ for some $i$ and   $P \not\in  C_j$ for all $i \neq j$. 
Then $\partial_P(\partial_{x_i}(\zeta)) = 1$ (\ref{complex_P}) and by the choice of $a$ and  $u_P$,
we have $[a(P)) = [u_P) = 1$.  
Suppose that $P \in C_i \cap C_j$ for some $i \neq j$. Suppose $i < j$. Then by the choice of 
$a$ and $u_P$ we have $\partial_P(\partial_{x_i}(\zeta)) = [u_P) = [a(P))$.
Suppose $i > j$.  Then by the choice of $a$ and $u_P$ we have $\partial_P(\partial_{x_j}(\zeta)) = 
[u_P) = [a(P))$. Since 
$\partial_P(\partial_{x_i}(\zeta))  = \partial_P(\partial_{x_j}(\zeta))^{-1}$ (\ref{complex_P}),
we have $\partial_P(\partial_{x_i}(\zeta)) = [a(P))^{-1}$.
Thus $a$ has the required properties. 
  \end{proof}

Let $a \in A$ be as in (\ref{choice_of_a}). 
%
%
%
%
Let $L_1, \cdots , L_d$ be the irreducible curves in $\XX$
which are in the support of $a$ or  $\frac{a-1}{(\rho-1)^\ell}$.
Then the cyclic extension $L = [a) = F(\sqrt[\ell]{a})$ is unramified at any irreducible curve 
in $\XX$ which is not equal to $L_j$ for any $j$ (cf. \ref{epp}).

\begin{lemma}
\label{properties_of_a2}
Then $L_i \cap \PP_0 = \emptyset$ for all $i$. In particular $L_i \neq C_j$ for all $i, j$ and char$(\kappa(L_i)) \neq \ell$. 
\end{lemma}
  
\begin{proof} By the choice of $a$, $a \in A$ and $\frac{a-1}{(\rho-1)^\ell} \in A$  (\ref{choice_of_a}).
Hence $\PP_0 \cap L_i = \emptyset$ for all $i$. Since $\PP_0$ contains at least one point from each $C_j$,
   $L_i \neq C_j$ for all $i$ and $j$.
Since supp$_\XX(\ell) \subset \{ C_1, \cdots , C_r  \}$,  char$(\kappa(L_i)) \neq \ell$ for all $i$.  
\end{proof}

   Let $\PP_1 \subset \cup_j L_j$ be a  finite set of closed points of $\XX$ consisting of  $L_i \cap L_j$ for $i \neq j$, 
 $L_i \cap C_j$, one point from each $L_i$.
  Since  $L_i \cap \PP_0 = \emptyset$ for all $i$ (\ref{properties_of_a2}), 
    $\PP_0 \cap \PP_1 = \emptyset$. 
 
Let $\PP  = \PP_0 \cup \PP_1$ and  
 $B$ be the semi local ring at $\PP$ on $\XX$. 
For each $i$ and $j$, let $\pi_i \in B$ be a prime defining $C_i$ and $\delta_j \in B$ a prime defining $L_j$.

 \begin{lemma}
 \label{local_choice_of_b}  
 For each $P \in C_i \cap \PP_1$, let
 $n_i^P$ be a positive integer. 
 Then for each $i$, $1 \leq i \leq r$, there exists   $b_{i} \in B/(\pi_i) \subset \kappa(C_i)$ such that \\
 $\bullet$  $\partial_ {C_i}(\zeta) = [ a( C_i),    b_i )$ \\
 $\bullet$  $\nu^i_P(b_{i}) = 1$ for all  $P \in  C_i \cap \PP_0$, $1 \leq i \leq r_1$  \\
 $\bullet$ $\nu^i_P(b_{i}) = \ell - 1 $ for all  $P \in  C_i \cap \PP_0$, 
 $ r_1 + 1 \leq i \leq r$  \\
$\bullet$    $\nu^i_{P}(b_{i} - 1) \geq n^P_i$ for all $P \in \PP_1  \cap C_i$ for all $i$. 
 \end{lemma}
 
 \begin{proof}  Let  $ 1 \leq i \leq r$. 
  Let  $\beta_{x_i} = \partial_{x_i}(\zeta) \in H^2(\kappa(C_i), \Z/\ell(1))$
 and $a_i = a(C_i)$.

 Suppose $1 \leq i \leq r_1$. By  (\ref{choice_of_a}),
 $\partial_P(\beta_{x_i}) = [ a_i(P) )$ for all $P \in C_i  \cap \PP_0$.
 If   $P \not\in \PP_0$, then   $\partial_P(\beta_{x_i}) = 0$ for all $i$ (\ref{complex_P}). 
 By the assumption, $\beta_{x_i} \otimes \kappa(C_i)_\nu = 0$ for all real places $\nu$ of $\kappa(C_i)$. 
 Thus,   by (\ref{global_field}),  there exists $b_{i} \in  \kappa( C_i)^*$ such that 
     $\beta_{x_i} = [ a_i,    b_{i} )$, with $\nu^i_P(b_{i}) = 1$ for all  $P \in  C_i \cap \PP_0$  
and     $\nu^i_{P}(b_{i} - 1) \geq n^P_i$ for all $P \in  C_i \cap \PP_1$. 
In particular $b_i$ is regular at all $P \in C_i \cap \PP$ and hence  $b_i \in B/(\pi_i)$.

Suppose $ r_1 + 1 \leq i \leq r$.
Let $P \in C_i \cap \PP_0$. 
 Since $\partial_P(\beta_{x_i}) = [ a(P) )^{-1} $ for all $P \in C_i \cap \PP_0 $ (\ref{choice_of_a}), 
 $\partial_P(\beta_{x_i}^{-1}) = [a(P))$. Thus, as above, 
   by (\ref{global_field}),  there exists $c_{i} \in  B/(\pi_i)$ such that 
     $\beta_{x_i}^{-1} = [a_i,    c_{i} )$, with $\nu^i_P(c_{i}) = 1$ for all  $P \in  C_i \cap\PP_0$
and     $\nu^i_{P}(c_{i} - 1) \geq n^P_i$ for all $P \in  C_i \cap \PP_1$. 
Let $b_i = c_i^{\ell-1} \in B/(\pi_i)$.  Then $\beta_{x_i} = [a_i, b_i)$. Let $P \in C_i \cap \PP_1$. 
Since $c_i \in B/(\pi_i)$ and $\nu^i_P(c_i -1) \geq n^P_i$, 
it follows tat $\nu^i_P(b_i - 1) \geq n^P_i$. 
Thus $b_i$ has the required properties.
 \end{proof}

 Let $\delta = \prod \delta_j \in B$. For $1 \leq i \leq r$, 
   let  $\overline{\delta}(i) \in B/(\pi_i)$ be the image of $\delta$.
   Let  $d$ be an integer greater than  $\nu_P^i(\overline{\delta}(i)) + 1$ for all  $i$ and 
   $P \in C_i \cap \PP$.

 \begin{lemma}
 \label{choice_of_b}   
 Let   $b_i \in B/(\pi_i) $ be as in (\ref{local_choice_of_b}) for  $n_i^P = d$ for all $P \in C_i \cap \PP$. 
 Then  there  exists $b \in B$  such that \\
 $\bullet$ $b = b_i$ modulo $\pi_i$  for all $i$\\
 $\bullet$  $b = 1$ modulo $\delta_j$ for all  $j$ \\
 $\bullet$ $b$ is a unit at all $P \in \PP_1$.
 \end{lemma}

\begin{proof}  
For $ 1\leq i \leq r$,  let $I_i = (\pi_i) \subset B$ and 
$I_{r+1} = ( \delta)  \subset B$.
Clearly the gcd$(\pi_i, \pi_j) = 1$  and 
gcd$(\pi_i,  \delta) = 1$ for all $1 \leq i < j \leq r$.
 For $1 \leq i < j \leq r$, 
$I_{ij} = I_i + I_j$ is either maximal ideal or equal to $B$. For $1 \leq i \leq r$,  we have $I_{i(r+1)} = (\pi_i, \delta)$.
Since $L_s \cap \PP_0 = \emptyset$ for all $s$, $(\delta_s,  \pi_i, \pi_j) = A$ for all $1 \leq i < j \leq r$ and for all $s$.
Thus the ideals $I_{ij}$, $ 1 \leq i < j \leq r+1$, are coprime. 
Let $b_{r + 1} = 1 \in B/(I_{r  +1})$.

Let $1 \leq i < j \leq r$.  Suppose $(\pi_i, \pi_j) \neq B$. Then $(\pi_i, \pi_j)$ is a maximal 
ideal of $B$ corresponding to a point $P \in C_i \cap C_j$. Since $P \in \PP_0$, 
by the choice of $b_i$  and $b_j$ (cf. \ref{choice_of_b}), we have $\nu^i_P(b_i) = 1$, $\nu^i_P(b_j) 
= \ell-1$ and hence $b_i = b_j = 0 \in B/(\pi_i, \pi_j) = B/I_{ij}$.  

Suppose $I_{i(r +1)} \neq B$ for some $1 \leq i \leq r$.  Then we claim that 
$b_i = 1 \in B/I_{i(r+1)}$. 
For each $P \in L_j \cap C_i$, let $M_P$ be the maximal ideal of $B$ at $P$. 
Since $\XX$ is regular and $C_i$ is  regular on $\XX$, we have 
$M_P = (\pi_i,  \pi_{i,P})$ for some $\pi_{i,P} \in M_P$ and the image of $\pi_{i,P}$ in $B/(\pi_i)$ is
a parameter at the discrete valuation $\nu_P^i$. Since $ d > \nu_P^i(\overline{\delta}(i)) $, 
we have $(\pi_i,  \prod \pi_{i,P}^d) \subset (\pi_i, \delta) = I_{i(r+1)}$. 
Since $B/(\pi_i, \prod \pi_{i,P}^d) \simeq \prod_P B/(\pi_i, \pi_{i,P}^d)$ and $\nu_P^i(b_i -1) \geq d$, 
we have $b_i = 1 \in B/(\pi_i, \prod \pi_{i,P}^d)$. Since  $B/I_i + I_{r +1}$ is a quotient of $B/I_i + (\prod_P \pi_{i,P})^d$, 
it follows that  $b_i = b_{r +1} = 1 \in B/I_i + I_{r+1} = B/I_{i(r+1)}$. 
 
Thus, by (\ref{saltman_exact_seq}), there exists $b \in B$ such that $b = b_i \in B/(\pi_i)$ for all $i$ 
and $b = 1 \in B/I_{r+1}$.   Since  $I_{r+1} =  ( \delta)  \subset (\delta_j)$  and $b =  1 \in B/(\delta)$, 
we have $b = 1 \in A/(\delta_j)$ for all $j$.  Let $P \in \PP_1$.  Then $P \in L_j$ for some $j$. 
Since $b = 1 \in B/(\delta_j)$, $b$ is a unit at $P$. 
Thus  $b$ has all the required properties. 
\end{proof}

 \begin{lemma}
 \label{choice_of_alpha} Let $a$ be as in (\ref{choice_of_a}) and $b$ as in (\ref{choice_of_b})
 and $\alpha = [a, b )$.    Then $\alpha$ is unramified at all $C_i$,  $L_j$ 
and at all   $Q \in \PP_1 $.
 Further $\partial _{C_i}(\zeta)$ is the specialization of $\alpha$ at 
$C_i$  for all $1 \leq i \leq r$.
 \end{lemma}
 
 \begin{proof} Since $[a )$ is unramified at $C_i$   (\ref{choice_of_a}) 
 and $b$ is a   unit at $C_i$ for all $i$ (\ref{choice_of_b}), $\alpha$ is  unramified at $C_i$
 and the specialization of $\alpha$ at $C_i$ is $[a(C_i), b_i) = \partial_{C_i}(\zeta)$ (\ref{local_choice_of_b}, 
 \ref{choice_of_b}). 
 Since char$(\kappa(L_j)) \neq \ell$ (\ref{properties_of_a2}) and  $b = 1$ modulo $\delta_j$ (\ref{choice_of_b}), 
 $b$ is an $\ell^{\rm th}$ power in  $F_{L_j}$ and hence $\alpha \otimes F_{L_j} = 0$.
 In particular $\alpha$   is unramified at $L_j$.
 
Let   $Q \in  \PP_1$. Then $b$ is  a unit at $Q$ (\ref{choice_of_b}). 
Let $x$ be a dimension one point of Spec$(B_Q)$. 
Then $b$ is a unit at $x$. If  $[ a )$ is unramified at $x$, then $\alpha$ is unramified at  $x$.
Suppose $[ a )$ is ramified at $x$. Then, by the choice of $L_j$'s,  $x$ is the generic point of $L_j$ for some $j$ and hence 
$\alpha$ is unramified at $x$. Thus $\alpha$ is unramified at $Q$ (cf. \ref{2dim_unramified}).
 \end{proof}

\begin{prop}
\label{prop}
The triple $(\XX,  \zeta, [ a, b ))$ satisfies the assumptions \ref{assumptions2}.
 \end{prop}
 
 \begin{proof}  
By the choice of $\XX$, B1) of \ref{assumptions2} is satisfied.  
Let ram$_\XX(\alpha) = \{ D_1, \cdots , D_n \}$. 
Since $\alpha$ is unramified at all $C_i$ (\ref{choice_of_alpha}), B2) of \ref{assumptions2} is satisfied. 
Since supp$_\XX(\ell) \subset \{ C_1, \cdots , C_r \}$ and $D_i \neq C_j$ for all $i $ and $j$, char$(\kappa(D_i)) \neq \ell$ for all $i$
and hence  B3) of \ref{assumptions2} is satisfied.

Let $P \in D_j$ some $j$ with   char$(\kappa(P)) = \ell$. 
Since supp$_\XX(\ell) \subset \{ C_1, \cdots , C_r \}$,  $P \in C_i$ for some $i$.
Since $\alpha$ is unramified at all $Q \in \PP_1$ (\ref{choice_of_alpha}), $P \not\in \PP_1$.
Since $C_i \cap L_s \subset \PP_1$ for all $s$,   $P \not \in L_s$ for all $s$
and hence $\frac{a-1}{(\rho-1)^\ell} \in A_P$. Thus B4)  of \ref{assumptions2} is satisfied. 

Since   $\partial_{C_i}(\zeta)$ is the specialization of $\alpha$ at $C_i$ (\ref{choice_of_alpha}),
B5) of \ref{assumptions2} is satisfied. 

By the assumption on $\zeta$, B6) of \ref{assumptions2} is satisfied.
 If $\ell = 2$,  then, by the choice of $a$ (\ref{choice_of_a}), B7)  of \ref{assumptions2} is satisfied. 

 Let  $P \in C_i \cap C_j$ for some $i < j$. 
Then,  by the choice of $b_i$ and $b_j$ (\ref{local_choice_of_b}), 
we have $b_i = \overline{u}_j \overline{\pi}_j$ for some unit 
$u_j$ at $P$ and $b_j =  \overline{u}_i \overline{\pi}_i^{\ell-1}$ for some unit 
$u_i$ at $P$. Since $b = b_i$ modulo $\pi_i$ and $b = b_j $ modulo $\pi_j$, we have 
$b = v_{i}\pi_i^{\ell-1} + v_j \pi_j$ for some units $v_{i},  v_j$ at $P$. In particular $b$ is a regular
prime at $P$. 
Since   $[a )$ is unramified at $P$ (\ref{choice_of_a}) and $b$ being  a prime at $P$,  $\alpha$ is unramified at $P$ except possibly 
at $b$.  Thus there is at most one $D_s$ with $P \in D_s$ and such a $D_s$ is defined by 
$b = v_{i}\pi_i^{\ell-1} + v_j \pi_j$ for some units $v_{i},  v_j$ at $P$. In particular B8) of \ref{assumptions2} is satisfied.

\end{proof}

\begin{theorem}
\label{main}   Let $K$ be a global field or a local field and $F$  the function field of 
a curve over $K$.  Let $\ell$ be a prime not equal to the characteristic of $K$. 
Suppose that $K$ contains a primitive $\ell^{\rm th}$ root of unity. 
Let $\zeta \in H^3(F, \Z/\ell(2))$. Suppose that $\zeta \otimes_F(F\otimes_KK_\nu)$ is trivial 
for all real places $\nu$ of $K$.
Then    there exist $a, b, f \in F^*$ 
such that $\zeta = [a , b ) \cdot  (f) $. 
\end{theorem}

 \begin{proof} By (\ref{prop}), there exist $ a, b  \in F^*$   and 
 regular proper model $\XX$ of $F$ such that the triple $(\XX, \zeta, \alpha)$ satisfy the assumption \ref{assumptions2}.
  Thus, by (\ref{localglobal}), there exists $f \in F^*$ such that  $\zeta = \alpha \cdot (f) = [a, b) \cdot (f)$. 
  \end{proof}

\begin{cor}
\label{imaginary}
 Let $K$ be a global field   or a local field and $F$ the function field of 
a curve over $K$.  Let $\ell$ be a prime not equal to the characteristic of $K$.
Suppose that $K$ contains a primitive $\ell^{\rm th}$ root of unity. 
Suppose that either  $\ell \neq 2$ or   $K$ has no real orderings.
Then  for every element $\zeta \in H^3(F, \Z/\ell(3))$, there exist $a, b, c \in F^*$ 
such that $\zeta = [a , b ) \cdot  (c ) $. 
\end{cor}

  \section{Applications}
  
  In this section we given some applications of our main result to quadratic forms and  Chow group of 
  zero-cycles. 
  
  Let $K$ be a field of characteristic not equal to 2. Let $W(K)$ denote the Witt group of 
  quadratic forms over $K$ and $I(K)$ the fundamental ideal of $W(K)$ consisting of 
  classes of even dimensional forms (cf. \cite[Ch. 2]{Sc}). For $n \geq 1$, let $I^n(K)$ denote the 
  $n^{\rm th}$ power of $I(K)$.    For $a_1, \cdots , a_n \in F^*$, let $<<a_1, \cdots , a_n>>$ denote the
  $n$-fold Pfister form $<\!1, -a_1\!>\otimes  \cdots \otimes <\!1, -a_n\!>$ (cf. \cite[Ch. 4]{Sc}.

  \begin{theorem}
  \label{3fold} 
  Let $k$ be a totally imaginary  number field   and $F$ the function field of 
a curve over $k$. Then every element in $I^3(F)$ is represented by a 3-fold Pfister form. 
In particular if the class of a quadratic form $q$ is in $I^3(F)$ and dimension of $q$ is at least 
9, then $q$ is isotropic.  
  \end{theorem}
  
  \begin{proof}  Since every element in $H^3(F, \Z/2(3))$ is a symbol (\ref{imaginary}) and
  cd$_2(F) \leq 3$,    it follows from (\cite[Theorem 2]{AEJ}) that every 
  element in $I^3(F)$ is represented by a 3-fold Pfister form (see the proof of \cite[Theorem 4.1]{PS1}).  
  \end{proof}

  \begin{prop} 
\label{rank10-general}Let $F$ be a field of characteristic not equal to 2 with cd$_2(F) \leq 3$
  Suppose that every element in $H^3(F, \Z/2(3))$ is a symbol. 
  If $q$ is a quadratic form over $F$ of dimension at least 5 and $\lambda \in F^*$, then 
  $q  \otimes <\!1, -\lambda\!>$ is isotropic.  
  \end{prop}
  
  \begin{proof}  
Without loss of generality we assume that dimension of $q$ is 5.
By scaling we also assume that $q = <\!-a, -b, ab, c, d\!>$ for some $a, b, c, d \in F^*$.
Let $q' = <\!-a, -b, ab, c, d, -cd\!> \otimes<\!1, -\lambda\!>$.
Since $ <\!-a, -b, ab, c, d, -cd\!> \in I^2(K)$ (\cite[p. 82]{Sc}),  $q' \in I^3(F)$.  
Hence, by (\ref{3fold}),  $q'$ is represented by 3-fold Pfister form. Since $q' \otimes F(\sqrt{\lambda}) = 0$, 
$q' =  <\! 1, -\lambda\!>\otimes<\!1, \mu\!>\otimes<\!1, \mu'\!>$ for some  $\mu, \mu' \in F^*$ (cf. \cite[p. 45, Theorem 5.2]{Sc},  
\cite[p. 143, Corollary 1.5]{Sc} and \cite[p. 144, Theorem   1.4]{Sc}).
Since $H^4(F, \Z/2(4)) = 0$, $I^4(F) = 0$ (\cite[Corollary 2]{AEJ}), 
 we have $q' = -cd<\! 1, -\lambda\!>\otimes<\!1, \mu\!>\otimes<\!1, \mu'\!>$. 

Thus we have 
$$
\begin{array}{rcl}
<\!-a, -b, ab, c, d\!>\otimes<\!1, -\lambda\!>  & = & -cd<\!1, -\lambda\!>\otimes<\!1, \mu\!>\otimes<\!1, \mu'\!> + <\!1, -\lambda\!> \\
& = & -cd<\!1 - \lambda\!>\otimes<\!\mu, \mu', \mu\mu'\!>.
\end{array}
$$
 In particular  $<\!-a, -b, ab, c, d\!>\otimes<\!1, -\lambda\!>$ is isotropic (\cite[p. 34]{Sc}). 
\end{proof}

  \begin{cor} 
  \label{rank10}
  Let $K$ be a totally imaginary number field and  $F$ the function field a  curve over $K$.
  Let $q$ be a quadratic forms over $F$ of dimension at least  5. Let $\lambda \in F^*$.
  Then the quadratic form $q \otimes <\!1, -\lambda\!>$ is isotropic. 
  \end{cor}

\begin{proof}  Since $K$ is a totally imaginary number field and $F$ is a function field of a curve 
over $k$, we have $H^4(F, \Z/2(4)) = 0$. 
Since every element in $H^3(F,\Z/2(3))$ is a symbol (\ref{imaginary}),  $q\otimes <\!1, -\lambda\!>$ is
isotropic (\ref{rank10-general})
\end{proof}

  \begin{theorem} 
    \label{zero-cycles}
    Let $k$ be a totally imaginary number field and  $C$ a smooth projective geometrically integral 
  curve over $K$. Let $ \eta : X \to C$ be an admissible quadric fibration. If dim$(X) \geq 4$, 
  then    $CH_0(X)$ is a finitely generated abelian group.   
  \end{theorem}
 
  \begin{proof} Let $q$  be a quadratic form over $k(C)$ defining the generic fibre of $\eta : X \to C$.
  Let $N_q(k(C))$ be the subgroup of $k(C)^*$ generated by  $fg$ with $f, g \in k(C)^*$  represented by $q$. 
    Let $\lambda \in k(C)^*$. 
  Since dim$(X) \geq 4$,  the dimension of $q$ is at least 5.  Thus, by (\ref{rank10}), 
  $q \otimes <\!1, -\lambda\!>$ is isotropic. Hence $\lambda$ is a product of two values of $q$.
  In particular $\lambda \in N_q(k(C))$ and $k(C)^* = N_q(k(C))$. 
   
  Let $CH_0(X/C)$ be the kernel of the induced homomorphism $CH_0(X) \to CH_0(C)$. 
  Then,   by (\cite{CTSk}), $CH_0(X/C)$ is a sub quotient of the group $k(C)^*/N_q(k(C))$ and hence 
  $CH_0(X/C) = 0$.  In particular $CH_0(X)$ is isomorphic to a subgroup of $CH_0(C)$.
  Since, by a theorem of Mordell-Weil, $CH_0(C)$ is finitely generated, $CH_0(X)$ is finitely generated. 
    \end{proof}

\providecommand{\bysame}{\leavevmode\hbox to3em{\hrulefill}\thinspace}

\end{document}